\documentclass[12pt,a4paper, reqno]{amsart}
\usepackage[utf8]{inputenc}
\usepackage{amsmath}
\usepackage{amsfonts}
\usepackage{amssymb}
\usepackage{amsthm}
\usepackage{graphicx}
\usepackage{helvet}
\usepackage{amsfonts,amsmath,verbatim,amssymb,amscd,latexsym}
\usepackage{amsthm, calc, enumerate, amscd}
\usepackage[english]{babel}
\usepackage{hyperref}
\usepackage[margin=2.6cm]{geometry}

\newtheorem{theorem}{Theorem}

\newtheorem{lemma}{Lemma}

\begin{document}

\author{Mohd Harun,\,\,Saurabh Kumar Singh}

\title{Shifted convolution sum with weighted average : $GL(3) \times GL(3)$ setup}

\address{ Saurabh Kumar Singh \newline {\em Department of Mathematics and Statistics, Indian Institute of Technology, Kanpur, India; \newline  Email: skumar.bhu12@gmail.com
} }
\address{ Mohd Harun \newline {\em Department of Mathematics and Statistics, Indian Institute of Technology, Kanpur, India; \newline  Email: harunmalikjmi@gmail.com
} }
\date{\today}

\maketitle	
	
\begin{abstract}  
This article will prove non-trivial estimates for the average and weighted average version of general $GL(3) \times GL(3)$ shifted convolution sums by using the circle method.
\end{abstract}

\section{Introduction}
Let $f(n)$ and $g(n)$ denote two arithmetical functions. The associated Shifted Convolution Sum (SCS) is defined as follows:
\begin{align*}
\mathcal{L}(X,f,g): = \sum_{n \leq X} f(n) g(n+h), \hspace{1cm} h \in \mathbb{Z}\setminus\{0\}.
\end{align*}
In modern number theory, numerous prominent problems involve investigating the asymptotic behavior or obtaining non-trivial estimates for various SCSs. These problems include finding asymptotics or non-trivial bounds for the moments of $L$-functions, subconvexity, additive divisor problems, Chowla's conjecture, and more. In this section, our focus will be on concise literature directly related to the problems addressed in this paper.

In the past few decades, substantial efforts have been made to study the additive divisor problems that seek the asymptotic evaluation of the sum
\begin{align}\label{z1a}
\mathcal{L}(X,d_k,d_{\ell}) = \sum_{n\leq X} d_k(n) d_{\ell}(n+h),
\end{align}
where $d_k(n)$ represents the generalized divisor function, counting the number of ways of writing $n$ as the product of $k$ positive integers. When $k=\ell = 2$, this sum is associated with a well-known problem called the Binary Additive Divisor Problem, which has been extensively investigated in the past (see \cite{r21} for a historical survey). Interest in the sum $\mathcal{L}(X,d_2,d_{2})$ is significantly driven by its connection to the Riemann zeta function $\zeta(s)$. The behavior of $\zeta(s)$ on the critical line can be studied through the moments
\begin{align*}
M_q(T, \zeta) := \int_{0}^{T} \left|\zeta\left(\frac{1}{2} +it\right)\right|^{2q} dt.
\end{align*}
Asymptotic formulas for the second moment $M_1(T,\zeta)$ and even for the fourth moment $M_2(T,\zeta)$ have already been established (see \cite{r31}). Applying the approximate functional equation for $\zeta(s)$ simplifies the task of determining the asymptotic behavior of $M_2(T, \zeta)$ to that of finding the asymptotic behavior of the SCS $\mathcal{L}(X, d_2, d_2)$. However, finding asymptotic behavior or non-trivial bounds for the higher moments $M_q(T,\zeta)$ with $q\geq 3$ remains an open problem. In this context, studying SCSs $\mathcal{L}(X,d_k,d_{k})$ for $k \geq 3$ could shed light on the higher moments of $\zeta(s)$. Nevertheless, even for the simplest case of $k=3$ and $h=1$, the asymptotics remain unknown. The arithmetical function $d_k(n)$ can also be interpreted as the $n$-th coefficient of the Dirichlet series for the degree-3 $L$-function $\zeta^k(s)$. Consequently, $\mathcal{L}(X,d_3,d_{3})$ represents a particular instance of $GL(3) \times GL(3)$ SCSs. Regarding the average version of $\mathcal{L}(X,d_3,d_{3})$, some progress has been made by S. Baier, T.D. Browning, G. Marasingha, and L. Zhao \cite{r30}.

For $k=3,\, \ell =2$ or $k=2,\, \ell =3$, there is ample literature (see \cite{r25}, \cite{r26}, \cite{r27}, \cite{r28}) about these SCSs and the recent progress in this direction is due to Topacogullari \cite{r23}. In 1995, Pitt considered this problem of SCS involving normalized Fourier coefficients $\lambda_{f}(n)$ of any holomorphic cusp form $f$ for $SL(2, \mathbb{Z})$ and $d_3(n)$.  In his paper \cite{r6}, Pitt proved that, for $0 < r < X^{1/24}$, 
\begin{align*}
	\mathcal{L}(X,d_3,\lambda_f) =  \sum_{n\leq X} d_3(n)\, \lambda_f(rn-1)\ll_{\zeta^3,f,\epsilon}\, X^{1-1/72 +\epsilon}.
\end{align*} 
This is a non-trivial bound with a saving of size $X^{1/72}$ over the trivial bound $X^{1+\epsilon}$. After Pitt's seminal result, the problem of SCSs for $GL(3) \times GL(2)$ has garnered significant attention from number theorists. Notable developments related to this problem can be found in \cite{r11}, \cite{r5}, \cite{DUKE}, \cite{r19}, \cite{r10}, and \cite{r12}. Recently, Munshi approached this problem in a more general setting by replacing the specific $GL(3)$ Fourier coefficients $d_3(n)$ in Pitt's problem with $A_\pi(1,n)$, which are the normalized Fourier coefficients of any Hecke-Maass cusp forms $\pi$ for $SL(3,\mathbb{Z})$. In \cite{r7}, utilizing Jutila's variant of the circle method and the idea of factorizable moduli, Munshi proved (after a minor correction in his Lemma 11) the following result:
\[ \mathcal{L}(X,A_{\pi},\lambda_f) = \sum_{n=1}^\infty A_\pi(1,n) \lambda_f(n+h) V\left( \frac{n}{X}\right) \ll_{\pi, f, \epsilon}\, X^{1-1/26+\epsilon} ,\]    
where $V$ is a smooth compactly supported function and $1\leq |h| \leq X^{1+\epsilon}$ is an integer. Building upon Munshi's work, Xi \cite{r8} further investigated the same SCS and employed a similar variant of the delta method. However, Xi adopted a distinct approach to handle the involved character sum, resulting in an enhancement of Munshi's results. Specifically, he demonstrated:
\begin{align*}
\mathcal{L}(X,A_{\pi},\lambda_f) \ll_{\pi,f,\epsilon}\, X^{1-1/22 + \epsilon}.    
\end{align*}
Subsequently, taking into account Munshi's remark in \cite{r7} and adopting the same variant of the circle method (Jutila's), Sun \cite{r9} established a non-trivial estimate for the average version of the aforementioned SCS. For $r^{5/2}X^{1/4 + 7\delta/2} \leq H \leq X$ with $\delta >0$, she demonstrated
\[ \frac{1}{H}\sum_{h=1}^\infty V\left( \frac{h}{H}\right)\sum_{n=1}^\infty  \lambda(n)  \lambda_f(rn+h) W\left( \frac{n}{X} \right) \ll X^{1-\delta+ \epsilon}, \]
where $H > 0$, $W$ represents another compactly supported smooth function, $r\geq 1$ is an integer, and $\lambda(n)$ can be either $A_{\pi}(1,n)$ or $d_3(n)$. In \cite{r29}, the authors of the present article also established a non-trivial bound of similar strength for an extended version of Sun's problem. However, we employed a different variant of the circle method, namely the Duke, Friedlander, Iwaniec (DFI) circle method. Specifically, we investigated the weighted average version of the SCS for $GL(3) \times GL(2)$ with the weight being the $GL(2)$ Fourier coefficient with $h$-sum. For $X^{1/4+\delta} \leq H \leq X$ with $\delta >0$, we established
\[\frac{1}{H}\sum_{h=1}^\infty \lambda_g(h)\, V\left( \frac{h}{H}\right)\sum_{n=1}^\infty  A_{\pi}(1,n)  \lambda_f(n+h) W\left( \frac{n}{X} \right)\ll X^{1-\delta +\epsilon}, \]
where $\lambda_g(n)$ and $\lambda_f(n)$ are $GL(2)$ Fourier coefficients, and $A_{\pi}(1,n)$ is a $GL(3)$ Fourier coefficient.

In this article, we will establish non-trivial bounds for two variants of $GL(3) \times GL(3)$ SCSs. One is the average version, and the other is the weighted average version, which can be viewed as the $GL(2)$ twist over the $h$-sum. More precisely, we investigate the following SCSs, denoted by $\mathcal{D}(H,X)$ and $\mathcal{L}(H,X)$.
\begin{align*}
\mathcal{D}(H,X) = \frac{1}{H}\sum_{h=1}^\infty \lambda_f(h) W_1\left( \frac{h}{H}\right)\sum_{n=1}^\infty A_{\pi_1}(1,n) A_{\pi_2}(1,n+h) W_2\left( \frac{n}{X} \right),
\end{align*} and
\begin{align*}
\mathcal{L}(H,X) = \frac{1}{H}\sum_{h=1}^\infty W_1\left( \frac{h}{H}\right)\sum_{n=1}^\infty A_{\pi_1}(1,n) A_{\pi_2}(1,n+h) W_2\left( \frac{n}{X} \right),
\end{align*} where $A_{\pi_1}(1,n)$ and $A_{\pi_2}(1,n)$ are the normalized Fourier coefficients of $SL(3,\mathbb{Z})$ Hecke-Maass cusp forms $\pi_1$ and $\pi_2$, respectively, $\lambda_f(n)$ represents the normalized Fourier coefficients of a holomorphic or Hecke-Maass cusp form $f$ for $SL(2,\mathbb{Z})$. The functions $W_1$ and $W_2$ are chosen to be smooth and compactly supported. Using the Cauchy-Schwarz inequality and estimates from the Rankin-Selberg theory for the Fourier coefficients, we obtain the following trivial bound:
\begin{align*}
\mathcal{D}(H,X) \ll_{\pi_1,\pi_2,f,\epsilon}, X^{1+\epsilon}, \hspace{1cm} \mathcal{L}(H,X) \ll_{\pi_1,\pi_2,\epsilon}, X^{1+\epsilon}.
\end{align*}
Utilizing the DFI circle method once again, we establish the following two main results in this paper.

\begin{theorem}\label{th1}
	Let $\mathcal{D}(H, X)$ be as defined above. For any $\epsilon > 0$ and $X^{1/2 + \delta} \leq H < X^{1-\epsilon
    } $ with $\delta > 0$, we have
	\begin{align*}
		\mathcal{D}(H,X) \ll X^{1-\delta+\epsilon},
	\end{align*} where the implied constants are depending upon $\pi_1, \pi_2$,$f$ and $\epsilon$.
\end{theorem}

\begin{theorem}\label{th2}
Let $\mathcal{L}(H, X)$ be as defined above. For any $\epsilon > 0$ and $X^{1/2 + \delta} \leq H < X^{1-\epsilon} $ with $\delta > 0$, we have
	\begin{align*}
		\mathcal{L}(H,X) \ll X^{1-\delta+\epsilon},
	\end{align*} where the implied constants are depending upon $\pi_1, \pi_2$ and $\epsilon$.
\end{theorem}
\vspace{0.3cm}

\textbf{Motivation of our problem:} We have witnessed numerous remarkable developments in the literature concerning both particular and general versions of different $GL(3) \times GL(2)$ SCSs and their applications. However, there has been limited progress so far in dealing with the general $GL(3) \times GL(3)$ SCSs or its average version. This problem of estimating SCSs appears to be highly challenging in the $GL(3) \times GL(3)$ setting, even in particular cases. One reason for this difficulty is the lack of structural advantages when dealing with the involved character sum, leading to a little extra saving, as seen in the $GL(3) \times GL(2)$ case. Nevertheless, such sums arise naturally in studying subconvexity and other related problems for higher-degree $L$-functions. For instance, consider the moment problem for higher degree $L$-functions $L(F,s)$ corresponding to an automorphic form $F$. It involves finding asymptotics for the integral:
\begin{align*}
M_q(T, L_F) := \int_{0}^{T} \left|L\left(F, \frac{1}{2} +it\right)\right|^{2q} dt.
\end{align*} This problem boils down to the study of some higher-degree SCSs on applying the approximate functional equation for $L(F,s)$. Recently, Aggarwal, Leung, Munshi \cite{r20} encountered the $GL(3) \times GL(3)$ SCS while investigating the short second moment and subconvexity for $GL(3)$ $L$-functions. In their paper, they dealt with the following averaged shifted sum:
\begin{align*}
\sum_{h\sim H}\sum_{n\sim N} A_{\pi}(1,n) \overline{A_{\pi}(1,n+h)}e\left(\frac{th}{2\pi n}\right),
\end{align*} which resembles the $GL(1)$ twist over the $h$-sum and both the arithmetical functions are Fourier coefficients of a single Hecke-Maass cusp form $\pi$ for $SL(3, \mathbb{Z})$. In Theorem \ref{th1}, we are considering a more general $GL(3) \times GL(3)$ SCS, which can be viewed as the $GL(2)$ twist over the $h$-sum and both the arithmetical functions represent Fourier coefficients of two not necessarily identical $SL(3, \mathbb{Z})$ Hecke-Maass cusp forms.

In this article, we present a very detailed proof for Theorem \ref{th1}, and for Theorem \ref{th2}, we provide a brief sketch as the proof follows a similar line as Theorem \ref{th1}. Let's briefly discuss the main ideas involved in the proof of Theorem \ref{th1} to highlight some key points. We begin by applying the DFI delta method to separate the oscillations and then use summation formulae to obtain savings in the $h$, $n$, and $m$ sums. This transformation changes our object of study to a new one involving a complicated character sum and a fourfold integral. It is important to note that the savings obtained through the summation formulae are insufficient. To address this, we apply the Cauchy-Schwarz inequality and Poisson summation formula to the sums over $n_2$ and $h$. This step leads to different cases of non-zero frequency, enabling us to achieve a slight extra saving. This is one of our key points in the proof.
One of the main challenges is to obtain the desired square root cancellation in the character sum, which can be roughly written as: \begin{align*}
&\sum_{k (q)} \sum_{\ell (q)} \sideset{}{^\star} \sum_{a\;\mathrm{mod}\, q} e\left(\frac{\overline{a} k }{q}\right) S\left(\overline{a}, \ell; q\right) S\left(\overline{a}, m_{2}; q\right) \notag\\
&\times \sideset{}{^\star} \sum_{b\;\mathrm{mod}\; q} e\left(-\frac{\overline{b} k}{q}\right) S\left(\overline{b}, \ell; q\right) S\left(\overline{b}, m_{2}^\prime; q\right) e \left(\frac{k h}{ q} \right) e \left(\frac{\ell n_2}{ q} \right).
\end{align*}
Here, $S\left({a}, \ell; q\right)$ denotes the Kloosterman sum, and $e(z):= e^{2\pi iz}$ for any complex number $z$. The trivial bound for the above character sum is $q^8$, but we need to analyze it and reduce it to $q^4$ to obtain our desired result. To achieve square root cancellations in the zero frequency case where $n_2 = 0 = h$, we prove the multiplicativity of the character sum and carry out elementary calculations. The elementary analysis does not yield the desired square root cancellations for the non-zero frequency cases. In these cases, we adopt a different strategy. We first reduce our character sum to prime moduli and then prove the required cancellations. Finally, we conclude our desired result for Theorem \ref{th1} by estimating the involved summations.

\vspace{0.2cm}

\textbf{Notations:}
 By the notation $A \asymp B$ we mean that $ X^{-\epsilon} B \leq A \leq  X^{\epsilon}B$. By $A \approx B$ we mean $c_{1} A \leq B \leq c_{2} A$ for some positive real $c_{1},c_{2}$. By $A \sim B$ we mean that $B \leq A \leq 2B$. By the notation $A \ll B$, we mean that for any arbitrary small $\epsilon >0$, there is a constant $c >0$ such that $|A| \leq c X^{\epsilon}B$. At various places, we are hiding the factor of $X^{\epsilon}$ while using $\ll$ notation. The arbitrary small  $\epsilon>0$ may be different at different places. The implied constants may depend on the cusp forms $\pi_1, \pi_2$, $ f$ and $\epsilon$. By the notation $m|n^{\infty}$, we mean that the set of prime factors of $m$ is a subset of the set of prime factors of $n$.  \\

 \section{DFI Delta Method}\label{section3}

In this section, we will briefly explain one of the key steps to get our main shifted convolution sums of Theorems \ref{th1}, \ref{th2} into the proper setup to analyze it. We are using a version of the delta method due to Duke, Friedlander, and Iwaniec. More precisely, we will use the expansion $(20.157)$ given in Chapter 20 of the book by Iwaniec and Kowalski\cite{r1}. We are taking most of the details of this section from \cite[subsection 2.4]{r2}. Let $\delta : \mathbb{Z}\to \{0,1\}$ be defined by
\[
\delta(n,m)=
\begin{cases}
	1 &\text{if}\,\,n=m \\
	0 &\text{otherwise}
\end{cases}
. \]
Then for $n,m\in\mathbb{Z}\cap [-2L,2L]$, we have
\begin{equation}\label{delta}
	\delta(n,m)=\frac{1}{Q}\,\sum_{1 \leq q\leq Q}\frac{1}{q}\,\, \sideset{}{^\star} \sum_{a\bmod q}e\left(\frac{(n-m) a}{q}\right)\int_{\mathbb{R}} \psi(q,x)e\left(\frac{(n-m) x}{q	Q}\right) d x, 
\end{equation} where $Q=2L^{1/2}$. The function $\psi$ satisfies the following properties (see $(20.159)$ of \cite{r1}, and \cite[Lemma 15]{r2}). For any $A>1$ and $j \geqslant 1$, 
\begin{align}\label{Delta}
&\psi(q,x)=1+h(q,x),\,\,\,\,\text{with}\,\,\,\,h(q,x)=O\left(\frac{Q}{q }\left(\frac{q}{Q}+|x|\right)^A\right),\\
&\psi(q,x)\ll |x|^{-A},	
\end{align} 
\begin{align} \label{deri g}
	x^j \frac{\partial^j}{ \partial x^j} \psi(q, x) \ll \  \min \left\lbrace \frac{Q}{q}, \frac{1}{|x|} \right\rbrace \  \log Q.
\end{align}
 We will see that these properties of $\psi$ will help us to deal with the integral transform occurring after summation formulae. In particular, the second property implies that the effective range of integral given in equation \eqref{delta} is $[-L^{\epsilon}, L^{\epsilon}]$. It also follows that if $q \ll Q^{1- \epsilon}$ and $  x \ll Q^{- \epsilon} $, then  $ \psi(q, x)$ can
be replaced by $1$ at the cost of a negligible error term. If $ q \gg Q^{1- \epsilon}$, then we get $ x^j \frac{\partial^j}{ \partial x^j} \psi(q, x) \ll Q^{\epsilon}$, for any $ j \geqslant 1$. If $ q \ll  Q^{1- \epsilon}$  and $  Q^{- \epsilon} \ll |x| \ll   Q^{ \epsilon}$, then $ x^j \frac{\partial^j}{ \partial x^j} \psi(q, x) \ll Q^{\epsilon}$, for any $ j \geqslant 1$. Hence, we can view $  \psi(q, x)$ as a nice weight function in all cases.

\section{Sketch of the proof of Theorem 1}
This section will give a rough sketch for the proof of Theorem \ref{th1}. Our first step is separating the oscillations using the Kronecker delta symbol
$\delta(n, m)$ i.e. we can write 
\begin{align*}
	 \,\frac{1}{H}\sum_{h \sim H} \lambda_f(h)\,\sum_{n \sim X}  A_{\pi_1}(1,n) \, \sum_{m \sim Y}\, A_{\pi_2} (1,m)\,\,\delta(n+h,m).
\end{align*}
Next, we apply the DFI delta method from Section \ref{section3}, our main object of study becomes
\begin{align}\label{c1}
	\mathcal{D}(H,X) =& \frac{1}{HQ} \sum_{q\sim Q}\frac{1}{q}\,\,\, \sideset{}{^\star} \sum_{a\bmod q}\int_{\mathbb{R}}\, \psi(q,u)
	\,\sum_{h \sim H} \lambda_f(h)  \,e\left(\frac{ah}{q}\right)\,\notag\\ 
	&\times \sum_{n \sim X}  A_{\pi_1}(1,n)\, e\left(\frac{an}{q}\right)\, \sum_{m \sim Y}   A_{\pi_2} (1,m)\,e\left(\frac{-am}{q}\right)\,du,
\end{align}  
where $\psi(q, u)$ is a nice weight function (see Section \ref{section3}). We are taking the generic case i.e. $q \asymp Q = \sqrt{X}$, $Y \asymp X$,  $h \asymp H$, $n \asymp X$, $m \asymp Y$,  $1 \leq H \leq X$. Notice that applying the delta symbol results in a loss of size $X$. So, we need to save $X$ and a little extra in the sum given in equation \eqref{c1} to get our desired bound.

Next, we apply the summation formulae to the sum over $h,n$, and $m$. We have captured the calculations with more detail in Lemma \ref{lemma8}, Lemma \ref{lemma9}, and Lemma \ref{lemma10}. Applying the $GL(2)$-Voronoi summation formula from Lemma \ref{voronoi} to the sum over $h$, roughly we arrive at
\begin{align*}
	\sum_{h \sim H}\lambda_f (h)\,e\left(\frac{ah}{q }\right) = \frac{ H^{3/4}}{q^{1/2}}  \sum_{h \sim Q^2/H} \frac{\lambda_f(h)}{h^{1/4}}  e\left(-  \frac{h\overline{a}}{q}\right).
\end{align*}
So, our saving in this step will be
\begin{align*}
	\text{saving} = \sqrt{\frac{\text{Initial length of summation}}{\text{Dual length of summation}}} = \sqrt{\frac{H}{Q^2/H}} = \frac{H}{Q}.
\end{align*}
Next, we apply the $GL(3)$-Voronoi summation formula from Lemma \ref{gl3voronoi} to the sum over $n$ and $m$. Roughly, the $n$-sum (similarly the $m$-sum) will get transformed into
\begin{align*}
	\sum_{n \sim X}A_{\pi_1}(1,n)e\left(\frac{an}{q }\right) = \frac{X^{2/3}}{q} \sum_{n_{2}\sim Q^3/X}  \frac{A_{\pi_1}(n_{2},1)}{ n^{1/3}_{2}} S\left( \bar{a}, \pm n_{2}; q \right).
\end{align*} Our savings in the $n$- and $m$-sums will be  $X/Q^{3/2}$ and  $X/Q^{3/2}$, respectively. After applying the summation formulae, up to a negligible error term, our main object of study $\mathcal{D}(H, X)$ will get transformed into the following expression
\begin{align}\label{c2}
	\mathcal{D}(H,X) = \notag&\frac{X^{4/3}}{Q^{9/2}H^{1/4}}\sum_{q\sim Q} \,\,  \sum_{n_{2} \sim  Q^3/X}  \frac{A_{\pi_1}(n_{2},1)}{ n^{1/3}_{2}} \, \sum_{m_{2}\sim  Q^3/Y}  \frac{A_{\pi_2}(m_{2},1)}{ m^{1/3}_{2}}\,  \\
	\times &  \,\,\,\sum_{h\sim Q^2/H} \frac{\lambda_f(h)}{h^{1/4}}\,\, \mathcal{C}( n_2,m_2, h; q)\,
\end{align} where  the character sum is given by
\begin{align*}
	\mathcal{C}( n_2,m_2, h; q) = \sideset{}{^\star} \sum_{a\, mod\, q} e\left( \frac{\overline{a}h}{q}\right) S\left(\overline{a},  \pm n_{2}; q \right)  S\left( \overline{a},  \pm m_{2}; q \right) \rightsquigarrow  \, q \, \mathcal{C}_1(...).
\end{align*} Here $\mathcal{C}_1(...)$ is another character sum modulo $q$ in which we still need to show square root cancellation, which we will prove in Subsection \ref{f4}. So, we save $\sqrt{Q}$ in the $a$-sum. Our total savings so far is
\begin{align*}
	\frac{X}{Q^{3/2}} \times \frac{X}{Q^{3/2}} \times \frac{H}{Q} \times \sqrt{Q} = \frac{X H}{Q^{3/2}}.
\end{align*} Therefore, we have to save $Q^{3/2}/H$ and a little more in the sum given in equation \eqref{c2} to get our desired bound. In the next step, we apply the Cauchy-Schwarz inequality in the $h$- and $n_2$- sum, and we arrive at the following expression
\begin{align*}
	\mathcal{D}(H,X) \ll \frac{X^{7/6}}{Q^{5/2}H^{1/2}}\,\sum_{q\sim Q}\, \left(	\sum_{n_2 \sim Q^3/X}\,\sum_{h\sim Q^2/H}\left| \sum_{m_{2}\sim Q^3/Y}  \frac{A_{\pi_2}(m_{2},1)}{ m^{1/3}_{2}}\,
	\,\, \mathcal{C}_1(...)\right|^2 \,\right)^{1/2}.
\end{align*} By opening the modulus square, we apply the Poisson summation formula to the sums over $h$ and $n_2$. We need to save $Q^{3}/H^2$ in the resulting expression. In the zero frequency case $n_2 =0 = h$, we will save the whole length of the diagonal, i.e., $Q^3/Y$. This saving will be enough for our purpose if 
\begin{align*}
	\frac{Q^3}{Y} > \frac{Q^3}{H^2} \,\,\,\,\,\text{or equivalently} \,\,\,\,\,  X^{1/2} < H.
\end{align*}
In the non-zero frequency cases, there are three possibilities, i.e. 
\begin{align*}
	& n_2\neq  0, h \neq 0,\\
	& n_2= 0,  h \neq 0, \\
	& n_2 \neq 0, h =0 .
\end{align*} In the case $n_2 \neq 0$, $h \neq 0$, we are saving $ Q^4/XH$ which is good as long as $H < X^{1-\epsilon}$. Similarly, we are saving enough in the other cases to finally get our desired result of Theorem \ref{th1}. 

\section{Preliminaries}
This section will briefly give details about some important tools we use to prove our results. More precisely, we only require the Voronoi summation formula for the Fourier coefficients of a $SL(2,\mathbb{Z})$ and $SL(3,\mathbb{Z})$ Hecke-Maass forms and some standard estimates on Fourier coefficients of these forms.

\vspace{0.5cm}

\begin{lemma} \label{voronoi}
	Let $\lambda(n)$ be the nth Fourier coefficients of $SL(2, \mathbb{Z})$ Hecke-Maass cusp form and $h$ be a compactly supported smooth function on the interval $(0, \infty)$. We have
	\begin{equation} \label{varequation}
		\sum_{n=1}^\infty \lambda (n) e_q(an) h(n) = \frac{1}{q} \sum_{\pm} \sum_{n=1}^\infty \lambda(\mp n) e_q(\pm \overline{a}n) H^{\pm} \left( \frac{n}{q^2}\right),
	\end{equation}
	where $ a \overline{a} \equiv 1 (\textrm{mod} \  q)$, and 
	\begin{align*}
		&H^{-} (y)= \frac{- \pi}{\cosh( \pi \nu)} \int_0^\infty h(x) \left\lbrace  Y_{2i\nu } + Y_{-2i\nu }\right\rbrace \left( 4\pi \sqrt{xy}\right) dx, \\
		&H^{+} (y)= 4\cosh( \pi \nu) \int_0^\infty h(x)  K_{2i\nu }  \left( 4\pi \sqrt{xy}\right) dx,
	\end{align*} where $Y_{2i\nu }$ is a Bessel function of second kind.  
\end{lemma}
\begin{proof}
 See \cite{r16}.   
\end{proof}

\begin{lemma} \label{gl3voronoi}
	Let $\psi (x)$ be a compactly supported smooth function on $(0,\infty)$ with $\tilde \psi(s)$ as its Mellin transform. Let $\alpha_i,\,\,i = 1,2,3$ be the Langlands parameters and $A_{\pi_1}(n,m)$ be the $(n,m)$-th Fourier coefficient of a Maass form $\pi_1(z)$ for $SL(3,\mathbb{Z})$. Then we have
	\begin{align} \label{GL3-Voro}
		& \sum_{n=1}^{\infty} A_{\pi_1}(n,m)\, e\left(\frac{an}{q}\right) \psi(n) \\
		\nonumber & =q  \sum_{\pm} \sum_{n_{1}|qm} \sum_{n_{2}=1}^{\infty}  \frac{A_{\pi_1}(n_{1},n_{2})}{n_{1} n_{2}} S\left(m \overline{a}, \pm n_{2}; mq/n_{1}\right) \, G_{\pm} \left(\frac{n_{1}^2 n_{2}}{q^3 m}\right),
	\end{align} 
	where $G_{\pm}(x)$ is an integral transform depending on $\alpha_1, \alpha_2, \alpha_3$, $(a,q)=1$ with $\overline{a}$ is the multiplicative inverse modulo $q$, and  $$S(a,b;q) = \sideset{}{^\star}{\sum}_{x \,\rm mod \, q} e\left(\frac{ax+b\overline{x}}{q}\right), $$
	is the Kloostermann sum.
\end{lemma}
\begin{proof}
	See \cite{r14}.
\end{proof}

The following lemma gives an asymptotic expansion for the integral transform $G_{\pm}(x)$.
\begin{lemma} \label{GL3oscilation}
	Let $G_{\pm}(x)$ be as above,  and  $\psi(x) \in C_c^{\infty}(X,2X)$. Then for any fixed integer $K \geq 1$ and $xX \gg 1$, we have
	\begin{equation*}
		G_{\pm}(x)=  x \int_{0}^{\infty} \psi(y) \sum_{j=1}^{K} \frac{c_{j}({\pm}) e\left(3 (xy)^{1/3} \right) + d_{j}({\pm}) e\left(-3 (xy)^{1/3} \right)}{\left( xy\right)^{j/3}} \, \mathrm{d} y + O \left((xX)^{\frac{-K+5}{3}}\right),
	\end{equation*}
	where $c_{j}(\pm)$ and $d_{j}(\pm)$ are some  absolute constants depending on $\alpha_{i}$,  $i=1,\, 2,\, 3$.  
\end{lemma}
\begin{proof}
	See  \cite{XL}.
\end{proof}

The following lemma gives the Ramanujan bound for $A_{\pi_1}(m,n)$ on average (see \cite{r17}).

\begin{lemma} \label{ramanubound}
	We have 
	$$\mathop{\sum \sum}_{m^2 n \leq X} \vert A_{\pi_1}(m,n)\vert ^{2} \ll_{\pi_1, \epsilon} \, X^{1+\epsilon}.$$
\end{lemma}

We also use the following lemma in our proof. 
\begin{lemma}\label{RMB}
We have
\begin{align*}
\mathop{\sum}_{ n \sim X} \vert A_{\pi_1}(m, n)\vert ^{2} \ll_{\pi_1, \epsilon} \, m^{2\theta + \epsilon}\,X,
\end{align*}
where $\theta \leq 5/14$.
\end{lemma}
\begin{proof}
From the Hecke relation, we have
\begin{align*}
A_{\pi_1}(m, n) = \sum_{d | (m, n)}\,\mu(d) A_{\pi_1}\left(\frac{m}{d}, 1\right) A_{\pi_1}\left(1, \frac{n}{d}\right).   
\end{align*}
We can write
\begin{align*}
\mathop{\sum}_{ n \sim X} \vert A_{\pi_1}(m, n)\vert ^{2} &=  \mathop{\sum}_{ n \sim X} \left|\sum_{d | (m, n)}\,\mu(d) A_{\pi_1}\left(\frac{m}{d}, 1\right) A_{\pi_1}\left(1, \frac{n}{d}\right)\right|^2 \\
&\leq\, \left|\sum_{d | m}\,\mu(d) A_{\pi_1}\left(\frac{m}{d}, 1\right)\right|^2\, \mathop{\sum}_{ n \sim X/d}\, |A_{\pi_1}\left(1, n\right)|^2.
\end{align*}
Using the individual bound $A_{\pi_1}(m, n) \ll (mn)^{\theta + \epsilon}$ for $A_{\pi_1}({m}/{d}, 1)$, where $\theta \leq 5/14$ is the bound towards the Ramanujan conjecture on $GL(3)$ (see \cite{KM}) and using Lemma \ref{ramanubound} for the $n$-sum, we get
\begin{align*}
\mathop{\sum}_{ n \sim X} \vert A_{\pi_1}(m, n)\vert ^{2} \ll_{\pi_1}\, X\, \sum_{d | m}\,\frac{|\mu(d)|^2}{d}\,\left(\frac{m}{d}\right)^{2(\theta+\epsilon)} \,\ll_{\pi_1}\, m^{2\theta + \epsilon}\,X.  
\end{align*}

\end{proof}

\begin{lemma} \label{poisson}
	{\bf Poisson summation formula}: Let $f:\mathbb{R } \rightarrow \mathbb{R}$ be any Schwartz class function. The Fourier transform 
	of $f$ is defined  as 
	\[
	\widehat{f}(y) = \int_{ \mathbb{R}} f( x) e(- x   y) dx,
	\] where $dx$ is the usual Lebesgue measure on $ \mathbb{R } $. 
	We have 
	\begin{equation*}
		\sum_{ n \in \mathbb{Z}  }f(n) = \sum_{m \in \mathbb{Z} } \widehat{f}(m). 
	\end{equation*} If $W(x)$ is any smooth and compactly supported function on $\mathbb{R}$, we have,
	\begin{align*}
		\sum_{n \in \mathbb{Z}  }e\left( \frac{an}{q}\right) W\left( \frac{n}{X}\right) = \frac{X}{q} \sum_{ m \in \mathbb{Z}  }\, \sum_{\alpha (\textrm{mod} \ q )}  e\left(\frac{( a + m) \alpha}{q} \right) \widehat{W} \left( \frac{mX}{q} \right). 
	\end{align*} 
\end{lemma}
\begin{proof}
	See  \cite[page 69]{r1}.
\end{proof}

\noindent The next lemma will prove that the greatest common divisor, on average, behaves like $1$.

\begin{lemma}\label{ll1}
 Let $X$ be any variable and $m \in \mathbb{Z}$, then we have
 \begin{align*}
     \sum_{n \leq X}\,(n, m)\, \ll_{\epsilon} \, X^{1+\epsilon},
 \end{align*}
where $(n, m)$ denotes the greatest common divisor of $m$ and $n$.
\end{lemma}
\begin{proof}
 Let $(n, m) = \ell$, then we can write
 \begin{align*}
\sum_{n \leq X}\,(n, m) = \sum_{\ell | m} \sum_{n \leq X/\ell}\,\ell \,\leq\, X \sum_{\ell | m}\,1 \,\ll_{\epsilon}\, X^{1+\epsilon}.  \end{align*}
\end{proof}

\noindent In the next lemma, we will prove the square root cancellation in a sum that is running over all the square full numbers up to a variable $X$.

\begin{lemma}\label{ll2}
For $X \geq 1$ be a variable, we have
\begin{align*}
    \mathop{\sum_{n \leq X}}_{n \,\text{is square full}}\,1\,\ll\, \sqrt{X}.
\end{align*}
\end{lemma}
\begin{proof}
  Since $n$ is square full, we can write $n = m^2_1 m_2$ such that $m_2 | m_1$.  We have
 \begin{align*}
  S =\,    \mathop{\sum_{n \leq X}}_{n \,\text{is square full}}\,1\, \leqslant \sum_{\substack{ m^2_1 m_2 \leqslant X \\ m_2 \mid m_1}} 1  \leq \,  \sum_{m_2 \leqslant X}\,\,\, \sum_{\substack{m_1 \leq \sqrt{\frac{X}{m_2}} \\ m_2 | m_1 }}\,1.
 \end{align*}
 Since $m_2 | m_1 $, we write $m_1 = m_2 m_3$.  We have
 \begin{align*}
S \leqslant  \, \sum_{m_2 \leqslant X}\,\,\, \mathop{\sum_{m_3 \leq \frac{1}{m_2}\sqrt{\frac{X}{m_2}}}}\,1 \,\leq \,\sqrt{X}\,\sum_{m_2 \leqslant X}\, \frac{1}{m_2}{\frac{1}{\sqrt{m_2}}}  \,\ll \, \sqrt{X}. 
 \end{align*}
 This proves our lemma.
\end{proof}

\vspace{0.5cm}
\section{Proof of Theorem \ref{th1}}
Our main sum of study is given  by
\begin{align*}
	\mathcal{D}(H,X) =&\,  \frac{1}{H}\sum_{h=1}^\infty \lambda_f(h) W_1\left( \frac{h}{H}\right)\sum_{n=1}^\infty  A_{\pi_1}(1,n)  A_{\pi_2} (1,n+h) W_2\left( \frac{n}{X} \right)\\
 =&\, \frac{1}{H}\sum_{h=1}^\infty \lambda_f(h) W_1\left( \frac{h}{H}\right)\sum_{n=1}^\infty  A_{\pi_1}(1,n)  W_2\left( \frac{n}{X}\right)\\
	&\times \sum_{m=1}^\infty\, A_{\pi_2} (1,m)\, W_3\left( \frac{m}{Y}\right)\,\delta(n+h,m),
\end{align*} 
where $W_1, W_2$, $W_3$ are smooth and compactly supported functions with support in the interval $[1,2]$. By using the expression for the delta symbol $\delta(n,m)$ given in equation \eqref{delta}, we can rewrite the above sum as
\begin{align}\label{g1}
	\mathcal{D}(H,X) =&\, \frac{1}{HQ}\sum_{q\leq Q}\frac{1}{q}\,\,\, \sideset{}{^\star} \sum_{a\bmod q}\int_{\mathbb{R}}A(u)\, \psi(q,u)\, \sum_{h=1}^\infty \lambda_f(h)  \,e\left(\frac{ah}{q}\right)\,e\left(\frac{hu}{qQ}\right) W_1\left( \frac{h}{H}\right)\notag\\ 
	&\times \sum_{n=1}^\infty  A_{\pi_1}(1,n)\, e\left(\frac{an}{q}\right)\,e\left(\frac{nu}{qQ}\right) W_2\left( \frac{n}{X}\right)\notag\\
	&\times \sum_{m=1}^\infty   A_{\pi_2} (1,m)\,e\left(\frac{-am}{q}\right)\,e\left(\frac{-mu}{qQ}\right)  W_3\left( \frac{m}{Y}\right)du + O(X^{-2023}).
\end{align} 
Here $Y = X+H \asymp X$ and $A$ is a smooth and compactly supported function, supported on the interval $[-2X^{\epsilon}, 2X^{\epsilon}]$ such that $A(u)=1, \,\,\forall\, u \in [-X^{\epsilon}, X^{\epsilon}] $ with $A^{(j)}(u) \leq 1$.

\subsection{Applying the Voronoi summation formulae}
In this subsection, we will analyze our main sum given in equation \eqref{g1} by applying the Voronoi summation formulas to the involved summations.
We consider the case when $f$ is a Hecke-Maass cusp form for $SL(2,\mathbb
Z)$. There are almost similar, even simpler calculations for the case of holomorphic cusp forms.  Let us further rewrite the sum $\mathcal{D}(H,X)$ as
\begin{align}\label{A17}
	\mathcal{D}(H,&X)= \frac{1}{HQ}\sum_{q\leq Q}\frac{1}{q}\,\,\, \sideset{}{^\star} \sum_{a\bmod q}\int_{\mathbb{R}}A(u)\, \psi(q,u)\,\Lambda_1 (...)\,\Lambda_2 (...)\,\Lambda_3 (...)\,du + O(X^{-2023}),
\end{align} where
\begin{align}\label{B1}
	\Lambda_1 (...) = \sum_{h=1}^\infty \lambda_f(h)  \,e\left(\frac{ah}{q}\right)\, W_1\left( \frac{h}{H}\right) e\left(\frac{hu}{qQ}\right),
\end{align}
\begin{align}\label{B2}
	\Lambda_2(...) = \sum_{n=1}^\infty  A_{\pi_1}(1,n)\, e\left(\frac{an}{q}\right)\,W_2\left( \frac{n}{X}\right) e\left(\frac{nu}{qQ}\right),	
\end{align}
\begin{align}\label{B3}
	\Lambda_3(...) = \sum_{m=1}^\infty   A_{\pi_2}(1,m)\,e\left(\frac{-am}{q}\right)\,W_3\left( \frac{m}{Y}\right) e\left(\frac{-mu}{qQ}\right).
\end{align}

Now we will apply the $GL(2)$-Voronoi summation formula to $h$-sum. In fact, we have the following Lemma.

\begin{lemma}\label{lemma8}
	Let $\Lambda_1(...)$ be as given in equation \eqref{B1}, we have
	\begin{align*}
		\Lambda_1 (...) = \frac{{H^{3/4}}}{\sqrt q} \sum_{\pm} \sum_{h\ll q^2/H }\, \frac{\lambda_f( h)}{h^{1/4}}\, e\left(\pm \frac{\overline{a}h}{q}\right) \eth_1^{\pm}(h,u,q) + O(X^{-2023}),
	\end{align*} where the integral transform is given by
 \begin{align*}
 \eth_1^{\pm}(h,u,q) \, = \,   \int_0^\infty  V_1(x)\, e\left(\frac{Hxu}{qQ}\pm \frac{2 \sqrt{Hhx}}{q} \right)dx,
 \end{align*}
where $V_1$ denotes a new weight function depending upon $W_1$.
\end{lemma}
\begin{proof}  
	By applying the $GL(2)$-Voronoi summation formula given in Lemma \ref{voronoi} to the sum over $h$ in equation \eqref{B1}, we get
	
	\begin{align}\label{A5}
		\Lambda_1(...) = \frac{1}{q} \sum_{\pm} \sum_{h=1}^\infty \lambda_f(h)\, e\left(\pm \frac{\overline{a}h}{q}\right) \mathcal{W}^{\pm}_1 \left( \frac{h}{q^2}\right),
	\end{align}  where $\mathcal{W}^{\pm}_1$ represents the integral transform from Lemma \ref{voronoi}, i.e.

 \begin{align*}
		\mathcal{W}^{+}_1 \left( \frac{h}{q^2}\right)= 4\cosh( \pi \nu_1) \int_0^\infty  {W}_1\left( \frac{x}{H}\right) e\left(\frac{xu}{qQ}\right)  K_{2i\nu_1 }\left(\frac{4\pi \sqrt{hx}}{q}\right) dx.
	\end{align*}

 \begin{align*}
		\mathcal{W}^{-}_1 \left( \frac{h}{q^2}\right)= \frac{- \pi}{\cosh( \pi \nu_1)} \int_0^\infty {W}_1\left( \frac{x}{H}\right) e\left(\frac{xu}{qQ}\right) \, \left\lbrace  Y_{2i\nu_1 } + Y_{-2i\nu_1 }\right\rbrace \left(\frac{4\pi \sqrt{hx}}{q}\right) dx,
	\end{align*}
	
  We first deal with $\mathcal{W}^{+}_1 \left( \frac{h}{q^2}\right)$. By changing the variables $x \longrightarrow Hx$, and using the approximation of the Bessel function $Y_{\pm2 i\nu}(y),\, K_{\pm2i\nu}(y)$ (see appendix in \cite{r19}), we get

 \[ \mathcal{W}^{+}_1 \left( \frac{h}{q^2}\right) =\frac{H^{3/4}\sqrt{q}}{h^{1/4}}\,\eth_1^{+}(h,u,q), \] 
	where the integral transform $\eth_1^{+}(h,u,q)$ is defined as
\begin{align}\label{z1}
  \int_0^\infty  V_1(x)\, e\left(\frac{Hxu}{qQ}\pm \frac{2 \sqrt{Hhx}}{q} \right)dx.
 \end{align}

Here $V_1$ is a new smooth and compactly supported function. Notice that $\eth_1^{+}(h,u,q)$ is a sum of two integrals. Using integration by parts repeatedly, we conclude that the above integral $ \eth_1^{+}(h,u,q)$ is negligibly small unless $ h \ll \frac{q^2}{H}$.

	Similarly, in the case of $\mathcal{W}^{-}_1 \left( \frac{h}{q^2}\right)$, by following the same steps, we will get a similar expression for  $ \eth_1^{-}(h,u,q)$. 
Finally, by putting all these observations into equation \eqref{A5}, we get 
	\begin{align}\label{A9}
		\Lambda_1(...) = \frac{{H^{3/4}}}{\sqrt q} \sum_{\pm} \sum_{h \ll q^2/H } \frac{\lambda_f(h)}{h^{1/4}} e\left(\pm \frac{\overline{a}h}{q}\right) \eth_1^{\pm}(h,u,q) + O(X^{-2023}),
	\end{align} where the integral $\eth_1^{\pm}(h,u,q)$ is defined above in equation \eqref{z1}.
\end{proof}

Now we apply $GL(3)$-Voronoi summation formula to the $n$-sum and $m$-sum. In fact, we have the following Lemmas.
\begin{lemma}\label{lemma9}
	Let $\Lambda_2(...)$ be as given in equation \eqref{B2}, we have 
\begin{align*}
	\Lambda_2(...) = \frac{X^{2/3}}{q} \sum_{\pm} \sum_{n_{1}|q} \sum_{n^2_1n_{2}\ll N}  \frac{A_{\pi}(n_{1},n_{2})}{n^{-1/3}_{1} n^{1/3}_{2}} S\left( \overline{a}, \pm n_{2}; q/n_{1}\right)& \eth_2^{\pm}(n^2_1n_2,u,q)
\end{align*} where $N = \text{max}\left\{{q^3}/{X}, X^{1/2}u^3\right\}$, integral transform is given by
\begin{align*}
	\eth_2^{\pm}(n^2_1n_2,u,q) = \, \int_{0}^\infty V_2(z)\, e\left(\frac{Xuz}{qQ} \pm \frac{3(Xz n_1^2 n_2)^{1/3}}{q} \right)dz,
\end{align*}
where $V_2$ denotes a new weight function depending upon $W_2$.
\end{lemma}
\begin{proof}
	
	By applying the $GL(3)$-Voronoi summation formula given in Lemma \ref{gl3voronoi} to the $n$-sum in equation \eqref{B2}, we get
	\begin{align}\label{A10}
		\Lambda_2(...) = \,q\, \sum_{\pm} \sum_{n_{1}|q} \sum_{n_{2}=1}^{\infty}  \frac{A_{\pi}(n_{1},n_{2})}{n_{1} n_{2}} S\left(\overline{a}, \pm n_{2}; q/n_{1}\right) \, \mathcal{W}^{\pm}_2 \left(\frac{n_{1}^2 n_{2}}{q^3 }\right),
	\end{align} 

 Using Lemma \ref{GL3oscilation}, upto a negligible error term, the integral transform $\mathcal{W}^{\pm}_2 \left(\frac{n_{1}^2 n_{2}}{q^3 }\right)$ will reduce to the following

	\begin{align*}
		\mathcal{W}^{\pm}_2 \left(\frac{n_{1}^2 n_{2}}{q^3 }\right) = \frac{X^{2/3}}{q^2} (n_{1}^2 n_{2})^{2/3}\, \eth_2^{\pm}(n^2_1n_2,u,q)+ O(X^{-2023}),
	\end{align*} 
where
\begin{align*}
\eth_2^{\pm}(n^2_1n_2,u,q) = \, \int_{0}^\infty V_2(z)\, e\left(\frac{Xuz}{qQ} \pm \frac{3(Xz n_1^2 n_2)^{1/3}}{q} \right)dz.  
\end{align*}
	Here, $V_2$ is a new smooth and compactly supported function. Now, using integration by parts repeatedly, we get that the integral $\eth_2^{\pm}(n^2_1n_2,u,q)$  is negligibly small if
	\begin{align*}
		n^2_1n_2 \gg \left\{\frac{q^3}{X} + X^{1/2}u^3\right\} =: N.
	\end{align*}
	
	Hence, we can say that $\mathcal{W}^{\pm}_2 \left(\frac{n_{1}^2 n_{2}}{q^3}\right)$ is negligibly small if \,$n_1^2n_2 \gg {N}$. From equation \eqref{A10}, we get 
	
	\begin{align*}
		\Lambda_2(...)= \frac{X^{2/3}}{q}\sum_{\pm}  \sum_{n_{1}|q} \sum_{n^2_1n_{2}\ll N}  \frac{A_{\pi}(n_{1},n_{2})}{n^{-1/3}_{1} n^{1/3}_{2}} S\left(  \overline{a}, \pm n_{2}; q/n_{1}\right)\, &\eth_2^{\pm}(n^2_1n_2,u,q)+  O(X^{-2023}). 
	\end{align*}
 which is our desired result.
\end{proof}

\begin{lemma}\label{lemma10}
	Let $\Lambda_3(...)$ be as given in equation \eqref{B3},  we have 

\begin{align*}
	\Lambda_3(...) = \frac{Y^{2/3}}{q} \sum_{\pm} \sum_{m_{1}|q}\, \sum_{m^2_1m_{2}\ll M}  \frac{A_{\pi}(m_{1},m_{2})}{m^{-1/3}_{1} m^{1/3}_{2}} S\left( \overline{a}, \pm m_{2}; q/m_{1}\right)& \eth_3^{\pm}(m^2_1m_2,u,q)+  O(X^{-2023}),
\end{align*} where $M = \text{max}\left\{{q^3}/{Y}, Y^{1/2}u^3\right\} $, the integral transform is given by
\begin{align*}
	\eth_3^{\pm}(m^2_1m_2,u,q) = \, \int_{0}^\infty V_3(y)\, e\left(\frac{-Yuy}{qQ} \pm  \frac{3 (Yy m_1^2 m_2)^{1/3}}{q} \right)dy, 
\end{align*} where $V_3$ denotes a new weight function depending upon $W_3$.
\end{lemma}
\begin{proof}
The proof will follow from the same steps as in the above lemma.
\end{proof}

\vspace{0.2cm}
We conclude this subsection by combining the analysis done so far in the next lemma.
\begin{lemma}\label{p2}
	We have,
	\begin{align*}
		\mathcal{D}(H,X) = &\frac{X^{2/3}Y^{2/3}}{QH^{1/4}}\sum_{q\leq Q}\frac{1}{q^{7/2}}\, \,\, \sum_{\pm} \sum_{n_{1}|q} \sum_{n_{2} \ll N_0}  \frac{A_{\pi_1}(n_{1},n_{2})}{n^{-1/3}_{1} n^{1/3}_{2}} \notag\\
  &\times \sum_{\pm}\sum_{m_{1}|q} \sum_{m_{2}\ll M_0}  \frac{A_{\pi_2}(m_{1},m_{2})}{m^{-1/3}_{1} m^{1/3}_{2}}\,\,\sum_{\pm}\sum_{h\ll H_0} \frac{\lambda_f(h)}{h^{1/4}}\,\mathcal{C}^{\pm}(q;h,n_2)\,\notag\\
  &\times \int_{\mathbb{R}}A(u)\, \psi(q,u)\,\, \eth_1^{\pm}(h,u,q)\,\, \eth_2^{\pm}(n_1n^2_2,u,q)\,\, \eth_3^{\pm}(m_1m^2_2,u,q) \,du + O(X^{-2023}),
	\end{align*}  where $H_0 = q^2/H, N_0 = {N}/n_1^2$ and  $M_0 = {M}/m_1^2$, $A>0$ and the character sum, which we denote by $\mathcal{C}^{\pm}(q;h,n_2) $ is given as
 \begin{align*}
\, \sideset{}{^\star} \sum_{a\bmod q} e\left( \pm\frac{\overline{a}h}{q}\right) S\left( \overline{a},  \pm n_{2}; q/n_{1}\right)  S\left( \overline{a},  \pm m_{2}; q/m_{1}\right).     
 \end{align*}
\end{lemma}
\begin{proof}
	By putting the results from Lemma \ref{lemma8}, Lemma \ref{lemma9}, and Lemma \ref{lemma10} into the expression for $\mathcal{D}(H, X)$ given in equation \eqref{A17}, we will get our desired result. 
\end{proof}

As far as our analysis is concerned, all the eight terms of $\mathcal{D}(H, X)$ in the above lemma are of the same complexity. So we focus our attention on one such term, namely  
\begin{align}\label{A18}
	\mathcal{D}_1(H,X) = &\frac{X^{2/3}Y^{2/3}}{QH^{1/4}}\sum_{q\leq Q}\frac{1}{q^{7/2}}\, \,\,  \sum_{n_{1}|q} \sum_{n_{2} \ll N_0}  \frac{A_{\pi_1}(n_{1},n_{2})}{n^{-1/3}_{1} n^{1/3}_{2}}\,\,\sum_{h\ll H_0} \frac{\lambda_f(h)}{h^{1/4}} \notag\\
 &\times \sum_{m_{1}|q} \sum_{m_{2}\ll M_0}  \frac{A_{\pi_2}(m_{1},m_{2})}{m^{-1/3}_{1} m^{1/3}_{2}}\, \mathcal{C}^{+}(q;h,n_2)\,\mathcal{I}(n_1^2n_2, m_1^2m_2,h),
\end{align} where
\begin{align*}
 \mathcal{I}(n_1^2n_2, m_1^2m_2,h) \,=\,  \int_{\mathbb{R}}A(u)\, \psi(q,u)\,\, \eth_1^{+}(h,u,q)\eth_2^{+}(n_1^2n_2,u,q) \eth_3^{+}(m_1^2m_2,u,q) \, du.  
\end{align*}

\subsection{Simplification of integrals} In this subsection, we will simplify and find a bound for the above integral transform $\mathcal{I}(n_1^2n_2, m_1^2m_2,h)$. We have the following lemma.

\vspace{0.2cm}

\begin{lemma}\label{lemma12}
	We have,
	\begin{align*}
		\mathcal{I}(n_1^2n_2, m_1^2m_2,h)\,\ll 
		\frac{q}{Q}.
	\end{align*}
\end{lemma} 
\begin{proof}
	
	By substituting the expressions of the integrals $\eth_1^{+}(h,u,q),\,\, \eth_2^{+}(n_1^2n_2,u,q)$ and $\eth_3^{+}(m_1^2m_2,u,q)$ given in Lemma \ref{lemma8}, \ref{lemma9} and \ref{lemma10}, respectively, into $\mathcal{I}(n_1^2n_2, m_1^2m_2,h)$, we get

	\begin{align}\label{M2}
		\mathcal{I}(n_1^2n_2, m_1^2m_2,h)
  =& \,\int_{\mathbb{R}} A(u)\, \psi(q,u)\,e\left(\frac{(Hx + Xz-Yy)u}{qQ} \right)\notag\\
		&\times \,\, \int_0^\infty\,V_{1}(x) \,e\left( \frac{2 \sqrt{Hhx}}{q}\right)\,\int_0^\infty\,\,V_{2}(z)\,\,   e\left( \frac{3(Xz n_1^2 n_2)^{1/3}}{q}\right) \notag\\ 
  &\times\int_0^\infty V_{3}(y)\,e\left( \frac{3(Yy m_1^2 m_2)^{1/3}}{q} \right)
		dy\,dz\,dx\,du.
	\end{align}
	
	First, we analyze the $u$-integral, which is given by
	\begin{align}\label{A40}
		\int_{\mathbb{R}}A(u)\, \psi(q,u)\,e\left(\frac{(Hx - Yy + Xz)u}{qQ} \right) du.
	\end{align}
	We divide our analysis into two cases depending upon the size of variable $q$. For the small $q$ case $q \ll Q^{1-\epsilon}$, we split the $u$-integral into two parts as follows 
	\begin{align*}
		\left(\int_{|u| \ll Q^{-\epsilon}} + 	\int_{|u| \gg Q^{-\epsilon}}\right) A(u)\, \psi(q,u)\,e\left(\frac{(Hx - Yy + Xz)u}{qQ} \right) du.
	\end{align*}
	
	For the first integral with $|u| \ll Q^{-\epsilon}$, we use the properties of $\psi(q,u)$ given in equation \eqref{delta} to replace $\psi(q,u)$ by 1 with a negligible error term. So essentially, we get
	\begin{align*}
		\int_{|u| \ll Q^{-\epsilon}}A(u)\, e\left(\frac{(Hx - Yy + Xz)u}{qQ} \right) du.
	\end{align*}
	
	Now, integrating by parts repeatedly, we get that the integral is negligibly small unless
	
	\begin{align}\label{A36}
		|Hx+Xz-Yy| \ll q\,Q\,Q^{\epsilon} \,\,\, i.e \,\,\,\,\,\, \left|\frac{Hx+Xz}{Y} - y\right| \ll \frac{q\,Q}{Y}Q^{\epsilon} .
	\end{align}
	
	For the second integral with $|u| \gg Q^{-\epsilon}$, integrating by parts the $u$-integral repeatedly and using the properties of involved functions, i.e.
	\begin{align*}
		\frac{\partial^j}{ \partial u^j} \psi(q, u) \ll \  \min \left\lbrace \frac{Q}{q}, \frac{1}{|u|} \right\rbrace \  \frac{\log Q}{|u|^j}\ll Q^{\epsilon j}, \hspace{0.5cm}	A^j(u) \ll Q^{\epsilon j},
	\end{align*} we will get the same restriction as above in equation \eqref{A36}. Also for the generic case i.e. $q \gg Q^{1-\epsilon}$, condition $\left|\frac{Hx+Xz}{Y} - y\right| \ll {q\,Q^{1+\epsilon}}/{Y} $ is trivially true.

	\vspace{0.3cm}
	Writing $\frac{Hx+Xz}{Y} - y = t$ with $|t| \ll {qQ^{1+\epsilon}}/{Y}$, we arrive at the following expression of the integral $\mathcal{I}(n_1^2n_2, m_1^2m_2,h)$
	
	\begin{align}\label{t2}
		\,&\int_{|t| \ll \frac{qQ^{1+\epsilon}}{Y}}\,\int_{0}^\infty\, V_{1}(x)\,e\left(\frac{2 \sqrt{Hhx}}{q}\right)\,\int_{0}^\infty\,\,V_{2}(z)\,V_{3}\left(\frac{Hx+Xz}{Y}-t\right)\\ &\,\,\,\,\notag\times  \  \,e\left(\frac{3(Xz n_1^2 n_2)^{1/3}}{q} + \frac{3 ({m^2_1m_2(Hx+Xz-Yt)})^{1/3}}{q} \right) dz\,dx \,dt\,\,+\, O(X^{-2023}).
	\end{align}
	Estimating the above integral trivially, we get
	\begin{align*}
		\mathcal{I}(n_1^2n_2, m_1^2m_2,h)\, \ll\, \frac{qQ^{1+\epsilon}}{Y} \asymp \frac{q}{Q} Q^{\epsilon}.
	\end{align*}
	Which is our desired result. 
\end{proof}
 
\subsection{Applying Cauchy-Schwarz inequality} In this subsection, we apply the Cauchy-Schwarz inequality to the sums over $n_2$ and $h$ in our main sum $\mathcal{D}_1(H, X)$ given in equation \eqref{A18} to get rid of one of the $GL(3)$ Fourier coefficients and the $GL(2)$ Fourier coefficient. Notice that $\mathcal{D}_1(H,X)$ is dominated by
\begin{align*}
&\,\frac{X^{4/3}}{QH^{1/4}}\,\sum_{q\leq Q}\frac{1}{q^{7/2}}\, \sum_{n_{1}|q} n_1^{1/3}\,\,\sum_{m_{1}|q} m_1^{1/3}\, \sum_{n_{2} \ll N_0}\sum_{h\ll H_0} \frac{|\lambda_f(h)|}{h^{1/4}}\frac{|A_{\pi_1}(n_{1},n_{2})|}{ n^{1/3}_{2}}\\  
 &\times  \left| \sum_{m_{2}\ll M_0}  \frac{A_{\pi_2}(m_{1},m_{2})}{ m^{1/3}_{2}}\,\, \mathcal{C}^{+}(q;h,n_2)\,\, \mathcal{I}(n_1^2n_2, m_1^2m_2,h)\right|.
\end{align*}
Now, applying the Cauchy-Schwarz inequality to the sums over $n_2$ and $h$, we arrive at
\begin{align}\label{A32}
	\mathcal{D}_1(H,X) \ll \, &\frac{X^{4/3}}{QH^{1/4}}\,\sum_{q\leq Q}\frac{1}{q^{7/2}}\, \sum_{n_{1}|q} n_1^{1/3}\, \,\,\Theta^{1/2}\,\Lambda^{1/2}\,\Omega^{1/2},
\end{align} where
\begin{align}\label{c31}
	\Theta = \sum_{n_{2} \ll N_0}  \frac{\left|\lambda_{\pi}(n_{2},n_{1})\right|^2}{n^{2/3}_{2}}, 
\end{align}

\begin{align}\label{c32}
	\Lambda= \sum_{h\ll H_0} \frac{|\lambda_f(h)|^2}{h^{1/2}} \ll H^{1/2}_0,
\end{align}
\begin{align}
	\Omega = \sum_{m_{1}|q} m_1^{2/3}\,\sum_{n_2 \ll N_0}\,\sum_{h\ll H_0}\left| \sum_{m_{2}\ll M_0}  \frac{A_{\pi_2}(m_{1},m_{2})}{ m^{1/3}_{2}}\mathcal{C}^{+}(q;h,n_2)\, \mathcal{I}(n_1^2n_2, m_1^2m_2,h)\right|^2.
\end{align} 

\subsection{Applying Poisson summation formula} To handle the contribution of $\Omega$ to $\mathcal{D}_1(H,X)$, we open the absolute valued square $\Omega$ and break the sums over $n_2$ and $h$ into dyadic blocks $n_2 \sim N_1, h\sim H_1$ by using dyadic partition of unity, where $N_1 \ll N_0$ and $H_1 \ll H_0$. By using smooth compactly supported functions $F_1$ and $F_2$, we smooth out the sums over $n_2$ and $h$. We arrive at the following expression
\begin{align*}
	\Omega \ll \,X^{\epsilon}\,\mathop{\mathop{\text{sup}}_{N_1 \ll N_0}}_{H_1 \ll H_0}\,\sum_{m_1|q}\,m^{2/3}_1 \,\sum_{m_{2}\ll M_0} \frac{|A_{\pi_2}(m_{1},m_{2})|}{ m^{1/3}_{2}}\,\,\sum_{m'_{2}\ll M_0}\frac{|A_{\pi_2}(m_{1},m'_{2})|}{ (m'_{2})^{1/3}}  \, \, |L(...)|, 
\end{align*} where
\begin{align}\label{A31}
	L(...) =& \sum_{n_2 \sim N_1} F_1\left(\frac{n_2}{N_1}\right)\, \sum_{h\sim H_0}\,F_2\left(\frac{h}{H_1}\right) \,\mathcal{C}^{+}(q;h,n_2)\,\overline{\mathcal{C}^{+}(q;h,n_2)}\notag\\ 
 &\times\, \mathcal{I}(n_1^2n_2, m_1^2m_2,h)
  \overline{\mathcal{I}(n_1^2n_2, m_1^2m'_2,h)}.
\end{align} 
Now changing the variable $n_2 \rightarrow \ell + n_2\,P$ and $h \rightarrow k + h\,q$, where $P = \frac{q}{n_1}$ and using the expression for the character sum $\mathcal{C}^{+}(...)$, we obtain
\begin{align*}
	L(...) 
	& = \sum_{k\;\mathrm {mod}\; q}\,\, \sum_{h\in \mathbb{Z}}F_2\left(\frac{k+h\,q}{H_1}\right)\, \,\sum_{\ell\;\mathrm {mod}\; q/n_1}\,\, \sum_{n_2 \in \mathbb{Z}}F_1\left(\frac{\ell+n_2\,P}{N_1}\right)\,\\
 & \hspace{10pt} \times \sideset{}{^\star} \sum_{a\;\mathrm{mod} \; q} e\left( \frac{\overline{a}k}{q}\right) S\left(\overline{a},   \ell; q/n_{1}\right)  S\left( \overline{a},   m_{2}; q/m_{1}\right)\,\\
 & \hspace{10pt} \times \sideset{}{^\star} \sum_{b\;\mathrm{mod} \; q} e\left( -\frac{\overline{b}k}{q}\right) S\left( \overline{b},   \ell; q/n_{1}\right)  S\left( \overline{b},   m'_{2}; q/m_{1}\right)\\
 & \hspace{10pt} \times \mathcal{I}(n_1^2(\ell+n_2P ), m_1^2m_2,k+hq)\, \overline{\mathcal{I}(n_1^2(\ell+n_2P), m_1^2m'_2,k+hq)}.
\end{align*}Poisson summation formula over $h$ and $n_2$ yields 
\begin{align*}
	L(...) =& \frac{H_1\,N_1}{q\,P}\, \sum_{h\in \mathbb{Z}}\,\,\sum_{k\;\mathrm {mod}\; q}\,e\left( \frac{k\,h}{q}\right)\, \,\sum_{n_2 \in \mathbb{Z}}\,\sum_{\ell\;\mathrm {mod}\; P}\,e\left( \frac{n_2\ell}{q/n_1}\right)\,\\& \times \sideset{}{^\star} \sum_{a\;\mathrm{mod} \; q} e\left( \frac{\overline{a}k}{q}\right) S\left( \overline{a},   \ell; q/n_{1}\right)  S\left( \bar{a},   m_{2}; q/m_{1}\right)\,\\& \times \sideset{}{^\star} \sum_{b\;\mathrm{mod} \; q} e\left( -\frac{\overline{b}k}{q}\right) S\left( \overline{b},   \ell; q/n_{1}\right)  S\left( \bar{b},   m'_{2}; q/m_{1}\right)\mathcal{J}(h,n_2),
\end{align*} where $\mathcal{J}(...)$ is denoting the integral transform, which is given by
\begin{align}\label{B13}
	\mathcal{J}(h,n_2) =& \int_{\mathbb{R}} \int_{\mathbb{R}}\notag F_1(t_1) F_2(t_2) \mathcal{I}(n^2_1N_1t_1,m^2_1m_2,H_1t_2) \overline{\mathcal{I}(n^2_1N_1t_1 ,m_1^2m'_2,H_1t_2)}\\
 & \times e\left(-\frac{n_2N_1t_1}{P}\right)\,e\left(-\frac{hH_1t_2}{q}\right)dt_1\,dt_2.
\end{align}We arrive at
\begin{align}\label{B15}
	\Omega \ll& \,X^{\epsilon}\,\mathop{\mathop{\text{sup}}_{N_1 \ll N_0}}_{H_1 \ll H_0}\, \frac{H_1\,N_1\,n_1}{q^2}  \,\sum_{m_{1}|q} m_1^{2/3}\,\sum_{m_{2}\ll M_0}\, \frac{|A_{\pi_2}(m_{1},m_{2})|}{ m^{1/3}_{2}}\,\,\sum_{m'_{2}\ll M_0}\frac{|A_{\pi_2}(m_{1},m'_{2})|}{(m'_{2})^{1/3}}\notag\\
 &\times \,\sum_{h\in \mathbb{Z}}\,\sum_{n_2 \in \mathbb{Z}} \, |\mathfrak{C}(q;h,n_2)|\,|\mathcal{J}(h,n_2)|,
\end{align} where the integral transform $\mathcal{J}(h,n_2)$ is given above and the character sum, which we denote by $\mathfrak{C}(q;h,n_2)$, is given by\begin{align}\label{B11}
	& \sum_{k\;\mathrm{mod} q}\,\, \sum_{\ell \;\mathrm{mod}\, q/n_1  } \, \,\,\sideset{}{^\star} \sum_{a\;\mathrm{mod} \; q} e\left( \frac{\overline{a} k }{q}\right) S\left( \overline{a},   \ell; q/n_{1}\right)  S\left( \overline{a},   m_{2}; q/m_{1}\right)\notag\\
	& \times \sideset{}{^\star} \sum_{b\;\mathrm{mod} \; q} e\left( -\frac{\overline{b} k }{q}\right) S\left( \overline{b},   \ell; q/n_{1}\right)  S\left( \overline{b}, m_{2}^\prime; q/m_{1}\right) e \left( \frac{k h}{ q} \right) e \left( \frac{\ell n_2}{ q/n_1} \right).
\end{align}
We see that, to bound $\Omega$, we need to find the appropriate bounds for the character sum $\mathfrak{C}(q;h,n_2)$ and the integral $\mathcal{J}(h,n_2)$. We will estimate them in the next two subsections.

\subsection{Analysis of integral  $\mathcal{J}(h,n_2)$} In this subsection, we will find a bound for the integral $\mathcal{J}(h,n_2)$ in the different cases. We have the following lemma.

\begin{lemma}\label{f1}
For the case $n_2=0$ and $h =0$, we have
	\begin{align*}
		\mathcal{J}(0,0) \ll 
		 \frac{q^3}{Q^2 (XN_1n^2_1)^{1/3}}. 
	\end{align*}
 Also, for the case $n_2 \neq 0$ and $h = 0$, we have
 \begin{align*}
 \mathcal{J}(0,n_2) \ll 
		 \frac{q^3}{Q^2\,\sqrt{HH_1}}.    
 \end{align*}
 In the remaining two cases, we have
 \begin{align*}
  \mathcal{J}(h,0), \,\mathcal{J}(h,n_2),\, \ll 
		\frac{q^{2}}{Q^{2}} .   
 \end{align*}
 Also, for the case $n_2 \neq 0, \, h\neq 0$, $\mathcal{J}(h,n_2)$ will be negligibly small unless
	\begin{align*}
		|n_2| \ll  \frac{(XN)^{1/3}}{n_1N_1}  := N_2  \hspace{0.5cm} \text{or}\,\,\,\,\,\,\, |h| \ll \frac{q}{H_1} =: H_2.
	\end{align*}
\end{lemma}
\begin{proof}
	On plugging the expression of integral $\mathcal{I}(n^2_1N_1t_1,m^2_1m_2,H_1t_2)$ from equation \eqref{t2} into equation \eqref{B13} and considering the $t_1$ integral, say $\mathcal{J}_{t_1}$, we have
	\begin{align}
		\mathcal{J}_{t_1} = \, \int_{\mathbb{R}}\notag F_1(t_1)\,\,   e\left( \frac{3(XN_1n_1^2)^{1/3}}{q} \left(z^{1/3} - z^{1/3}_1\right)t^{1/3}_1\right)\,e\left(-\frac{n_2N_1t_1}{P}\right)\, dt_1 .
	\end{align}
Now, for the zero frequency case $n_2 = 0, h = 0$, changing the variable $t_1 \longrightarrow t^3_1$ and applying integration by parts repeatedly, we get that the integral $\mathcal{J}_{t_1}$ will be negligibly small unless $$z - z_1 \, \ll \, \frac{q}{(XN_1n^2_1)^{1/3}}.$$
The integral $\mathcal{J}(h,n_2)$ given in equation \eqref{B13} is eight-fold. Estimating it for the case $n_2 =0 = h$ by using restrictions we got for the $z$-, $z_1$-, $t$-integrals and remaining integrals trivially, we get  
\begin{align*}
    \mathcal{J}(0,0) \,\ll\, \frac{q^3}{Q^2 (XN_1n^2_1)^{1/3}}.
\end{align*}
Again, considering the $t_2$ integral, say $\mathcal{J}_{t_2}$, we have
\begin{align*}
\mathcal{J}_{t_2} \, =\, \int_{\mathbb{R}}\notag F_2(t_2)\,\,   e\left( \frac{2\sqrt{HH_1}}{q} \left(x^{1/2} - x^{1/2}_1\right)t^{1/2}_2\right)\,e\left(-\frac{hH_1t_2}{q}\right)\, dt_1.     
\end{align*} We use the change of variable $t_2 \longrightarrow t^2_2$ and apply integration by parts repeatedly to the $t_2$-integral. We get that the integral $\mathcal{J}_{t_2} $ will be negligibly small unless
\begin{align*}
 x - x_1 \,\ll\, \frac{q}{\sqrt{HH_1}}.   
\end{align*}
For the case $n_2 \neq 0,\, h=0$, estimating our main integral $\mathcal{J}(h,n_2)$ using the restrictions we got for the $x$-, $x_1$-, $t$-integrals and remaining integrals trivially, we get
\begin{align*}
 \mathcal{J}(0,n_2) \,\ll\,   \frac{q^3}{Q^2\,\sqrt{HH_1}}.  
\end{align*}
Also, we will get the required bound for $\mathcal{J}(h,0)$ and $\mathcal{J}(h,n_2)$ when we estimate the integral $\mathcal{J}(...)$ trivially using Lemma \ref{lemma12}.\\

\noindent For the last part of the proof, applying integration by parts repeatedly on the integral transform $\mathcal{J}_{t_1}$, we will get that it will be negligibly small unless
	\begin{align*}
		n_2 \ll  \frac{(XN)^{1/3}}{n_1N_1} = N_2.
	\end{align*} 
 Similarly, on applying integration by parts repeatedly to the integral $\mathcal{J}_{t_2}$ given in equation \eqref{B13}, we will get that it will be negligibly small unless
	\begin{align*}
		h \ll \frac{q}{H_1} = H_2.
	\end{align*} This is our desired result.     
\end{proof}

\subsection{Analysis of character sum $\mathfrak{C}(q;h,n_2)$}\label{f4} In this subsection, we will analyze our character sum in different cases. In the case of zero frequency, we will prove that $\mathfrak{C}(q;h,n_2)$ is multiplicative in the next Lemma. Then we will estimate it in the same case by assuming $q = p^{\gamma}$ where $p$ is prime and $\gamma \in \mathbb{N}$. Next, we will estimate the character sum in the other three possible cases i.e.
\begin{align*}
	& n_2= 0,  h \neq 0 ,\\
	& n_2 \neq 0, h =0, \\
	& n_2\neq  0, h \neq 0.
\end{align*} 
The trivial bound for $\mathfrak{C}(q;h,n_2)$ is $q^8$. We aim to get the square root cancellation in the character sum $\mathfrak{C}(q;h,n_2)$ for all four possible cases. We have the following lemma.
From equation \eqref{B11}, by changing the variables $\overline{a} \longrightarrow a$ and $\overline{b} \longrightarrow b$,  we can rewrite the character sum as

\begin{align*}
\mathfrak{C}(q;h,n_2)
	=&\sum_{\ell \;\mathrm{mod}\,  q/n_1  } \,\, \sideset{}{^\star} \sum_{a\;\mathrm{mod} \; q}  S\left( a,   \ell; q/n_{1}\right)  S\left( a,   m_{2}; q/m_{1}\right)\\
	&\times \sideset{}{^\star} \sum_{b\;\mathrm{mod} \; q}  S\left( b,   \ell; q/n_{1}\right)  S\left( b, m_{2}^\prime; q/m_{1}\right)  e \left( \frac{\ell n_2}{ q/n_1} \right)\,\sum_{k\;\mathrm{mod}\, q} e\left( \frac{(a-b+h) k }{q}\right).
\end{align*} 
Executing the $k$ sum and opening the first and third Kloosterman sums, we get
\begin{align*}
\mathfrak{C}(q;h,n_2)
	=&\, q \, \sum_{\ell \;\mathrm{mod}\,  q/n_1 } \, \sideset{}{^\star} \sum_{a\;\mathrm{mod} \; q}\,\, \,\sideset{}{^\star} \sum_{\alpha\;\mathrm{mod} \; q/n_1}\,e\left(\frac{\alpha a+\overline{\alpha}\ell}{q/n_1}\right)\,S\left( a,   m_{2}; q/m_{1}\right)\,\\
	&\times \sideset{}{^\star} \sum_{\alpha_1\;\mathrm{mod} \; q/n_1}\,e\left(\frac{ \alpha_1 (a+h)+\bar{\alpha_1}\ell}{q/n_1}\right)\,S\left( (a+h),   m'_{2}; q/m_{1}\right)\, e \left( \frac{\ell n_2}{ q/n_1} \right)\\
=&\,q\,\sideset{}{^\star} \sum_{a\;\mathrm{mod} \; q}\,\, \,\sideset{}{^\star} \sum_{\alpha\;\mathrm{mod} \; q/n_1}\,e\left(\frac{\alpha a}{q/n_1}\right)\,S\left( a,   m_{2}; q/m_{1}\right)\,S\left( a+h,   m'_{2}; q/m_{1}\right)\\
	&\times \sideset{}{^\star} \sum_{\alpha_1\;\mathrm{mod} \; q/n_1}\,e\left(\frac{\alpha_1 a+ \alpha_1h}{q/n_1}\right)\,\sum_{\ell \;\mathrm{mod}\, q/n_1 }\,e \left( \frac{ (n_2+ \bar{\alpha_1}+\bar{\alpha})\ell}{ q/n_1} \right).
\end{align*}
By executing the $\ell$-sum and determining $ \alpha_1$ in terms of $ \alpha $, we get
\begin{align}\label{M1}
\mathfrak{C}(q;h,n_2)
	=&\,\frac{q^2}{n_1}\,\sideset{}{^\star} \sum_{a\;\mathrm{mod} \; q}\,\, \,\sideset{}{^\star} \sum_{\alpha\;\mathrm{mod} \; q/n_1}\,e\left(\frac{\alpha a- (\overline{\bar{\alpha}+n_2})(a+h)}{q/n_1}\right)  S\left( a,   m_{2}; q/m_{1}\right)\,\notag\\
	&\times \,S\left( a+h,   m'_{2}; q/m_{1}\right).
\end{align}

For the zero-frequency case $n_2 =0=h$, we can write
\begin{align*}
\mathfrak{C}(q; 0, 0)
  \leq &\,\, \frac{q^3}{n^2_1}\, \mathfrak{C}_1(q; 0, 0), \hspace{1cm} \text{where}   
\end{align*}
\begin{align}\label{b11}
\mathfrak{C}_1(q; 0, 0) \,= \, \sideset{}{^\star} \sum_{a\;\mathrm{mod} \; q}\,\,S\left( {a},   m_{2}; q/m_{1}\right)\,  S\left( {a}, m_{2}^\prime; q/m_{1}\right).
\end{align}
By the next lemma, we will prove that $\mathfrak{C}_1(q; 0, 0)$ is multiplicative.
\begin{lemma}\label{B30}
	Let $\mathfrak{C}_1(q; 0, 0)$ be given by equation \eqref{b11}. Let $ q = q_1 q_2$ with  $(q_1, q_2)= 1$ and $m_1 =m_1^{\prime }m_1^{\prime \prime}$ with $(m_1^{ \prime}, m_1^{\prime \prime}) = 1 $ and $m'_1|q_1,\, m''_1|q_2$. We write $q/m_1 = q_1/ m_1^\prime \times q_2/ m_1^{\prime \prime}$. Then we have 
	\begin{align*}
		\mathfrak{C}_1(q_1 q_2; 0, 0) =\mathfrak{C}_1(q_1; 0, 0) \mathfrak{C}_1(q_2;0, 0).
  \end{align*}
\end{lemma}
\begin{proof}
	Let $(c_1, c_2) =1$ and $S (a, b; c )$ denotes the Kloosterman sum. We will use the following properties.
	\begin{align}\label{B31}
		S (a, b; c_1 c_2 ) = S (a \overline{c_2}, b  \overline{c_2}; c_1  ) S (a \overline{c_1}, b  \overline{c_1}; c_2  )   \ \ \ \ \ \ \ \ \ \ \ \  \textrm{and,}
	\end{align}	
	\begin{align}\label{B32}
		S (am, b m; c)  =  S (a m^2, b ; c) =  S (a, b m^2; c)  \ \ \  \ \ \ \textrm{if} \ \ (m, c) = 1.
	\end{align}
	By writing $ a = a_1 q_2 \overline{q_2} + a_2 q_1 \overline{q_1} $ , we split the sum over $a$ and using equations \eqref{B31} and \eqref{B32}, we can rewrite the character sum $\mathfrak{C}_1(q;0,0)$ as
\begin{align*}
		& \sideset{}{^\star} \sum_{a_1\;\mathrm{mod}\, q_1} S\left( a_1 \,  \overline{q_2/m_1^{\prime \prime} }^2,   m_2; q_1/m_{1}^\prime\right)\,S\left( a_1 \,  \overline{q_2/m_1^{\prime \prime} }^2,   m_2^\prime; q_1/m_{1}^\prime\right)  \\
  &\times\,\, \sideset{}{^\star} \sum_{a_2\;\mathrm{mod}\, q_2}\,S\left( a_2 \,  \overline{q_1/m_1^{\prime} }^2,   m_2; q_2/m_{1}^{\prime \prime} \right) \, \times S\left( a_2 \,  \overline{q_1/m_1^{\prime} }^2,   m_2^\prime; q_2/m_{1}^{\prime \prime} \right).
	\end{align*} 
By doing the following change of variables $a_1 \longrightarrow a_1 (q_2/m''_1)^2$,\,$a_2 \longrightarrow a_2 (q_1/m'_1)^2$, we finally arrive at
	
 \begin{align*}
\mathfrak{C}_1(q;0,0) 
  =&\, \mathfrak{C}_1(q_1;0,0)\,\mathfrak{C}_1(q_2;0,0).
	\end{align*} 
 This is our desired result. 
\end{proof}
\vspace{0.2cm}

Now we prove the required bound for the zero frequency case $n_2 = 0 = h$ in the next lemma.

\begin{lemma}\label{B17}
	For the case $n_2 =0 =h$, let $q = q_1\,q_2$ such that $q_1|(m_1)^{\infty}$ and $(q_2, m_1) = 1$, we have
	\begin{align*}
		\mathfrak{C}(q; 0, 0) \ll \frac{q^5}{n^2_1\,m_1}, \hspace{0.4cm}  \text{under the condition} \hspace{0.4cm} m_2  \equiv m'_2 \;\mathrm{mod} \; q_2.
	\end{align*} 
\end{lemma}
\begin{proof}
	For the case $n_2 = 0 = h$, the character sum given in equation \eqref{M1} will become
 
	\begin{align*}
		\mathfrak{C}(q; 0, 0)
  \leq &\,\, \frac{q^3}{n^2_1}\, \,\sideset{}{^\star} \sum_{a\;\mathrm{mod} \; q}\,\,S\left( {a},   m_{2}; q/m_{1}\right)\,  S\left( {a}, m_{2}^\prime; q/m_{1}\right). 
 \end{align*} 
 We write $a$ as $a = a_1q_2\overline{q_2} + a_2q_1\overline{q_1}$. Using the properties of Kloosterman sums given in equations \eqref{B31} and \eqref{B32}, we can rewrite the character sum as
 \begin{align}\label{q1}
| \mathfrak{C}(q; 0, 0)| \, \leq \,\frac{q^3}{n^2_1}\,\, \mathfrak{C}(q_1)\,\mathfrak{C}(q_2), \hspace{1cm} \text{where}    
 \end{align}
 \begin{align*}
\mathfrak{C}(q_1) \, = \, \,\sideset{}{^\star} \sum_{a_1\;\mathrm{mod} \; q_1}\,\,S\left( {a_1} (\overline{q_2})^2,   m_{2}; q_1/m_{1}\right)\,  S\left( {a_1}(\overline{q_2})^2, m_{2}^\prime; q_1/m_{1}\right).
 \end{align*}
 and
  \begin{align*}
\mathfrak{C}(q_2) \, =\, \,\sideset{}{^\star} \sum_{a_2\;\mathrm{mod} \; q_2}\,\,S\left( {a_2} (\overline{q_1/m_1})^2,   m_{2}; q_2\right)\,  S\left( {a_2}(\overline{q_1/m_1})^2, m_{2}^\prime; q_2\right).
 \end{align*}
 Estimating $\mathfrak{C}(q_1)$ by using Weil's bound for the Kloosterman sums and summing over $a_1$ trivially, we get
 \begin{align*}
 \mathfrak{C}(q_1) \,\ll \, \frac{q^2_1}{m_1}\, \left( {a_1} (\overline{q_2})^2,   m_{2}, q_1/m_{1}\right)^{1/2}\,\left( {a_1} (\overline{q_2})^2,   m'_{2}; q_1/m_{1}\right)^{1/2}\,d\left(\frac{q_1}{m_1}\right)\,d\left(\frac{q_1}{m_1}\right).  
 \end{align*}
 Since we know that when we take an average over $m_2$ and $m'_2$, these greatest common divisors will behave like $1$ (see Lemma \ref{ll1}). So, we can write
 \begin{align}\label{q2}
     \mathfrak{C}(q_1) \,\ll \, \frac{q^2_1}{m_1}\,q^{\epsilon}.
 \end{align}
 To find the bound for $\mathfrak{C}(q_2)$, we first use the change of variables $a_2 \longrightarrow {a_2} ({q_1/m_1})^2$, getting
 \begin{align*}
 \mathfrak{C}(q_2) = \,\sideset{}{^\star} \sum_{\beta\;\mathrm{mod} \; q_2}\, \, \sideset{}{^\star} \sum_{\beta_1\;\mathrm{mod} \; q_2}\, e\left(\frac{\bar{\beta}m_2+\bar{\beta_1}m'_2}{q_2}\right)\,\sideset{}{^\star}\sum_{a_2\;\mathrm{mod} \; q_2}\,\,e\left(\frac{(\beta+ \beta_1) {a_2}}{q_2}\right).  
 \end{align*}
Now, we use its multiplicative property, which follows from Lemma \ref{B30}. It is enough to consider the case $q_2 = p^{\gamma}$, where $p$ is prime and $\gamma \in \mathbb{N}$.  We will deal with the case of $\gamma = 1$ and $\gamma > 1$ separately. We notice that the sum over $a_2$ is a Ramanujan sum. For $q_2 = p$, $\mathfrak{C}(q_2 = p)$ will become
 \begin{align*}
 = &\, \,\sideset{}{^\star} \sum_{\beta\;\mathrm{mod} \; p}\, \, \sideset{}{^\star} \sum_{\beta_1\;\mathrm{mod} \; p}\, e\left(\frac{\bar{\beta}m_2+\bar{\beta_1}m'_2}{p}\right)\,\sum_{d \mid (p, \beta+ \beta_1)} d \mu \left( \frac{p}{d}\right)\, \\
 =&\, \,\mu(p)\sideset{}{^\star} \sum_{\beta\;\mathrm{mod} \; p}\, \, \sideset{}{^\star} \sum_{\beta_1\;\mathrm{mod} \; p}\, e\left(\frac{\bar{\beta}m_2+\bar{\beta_1}m'_2}{p}\right)+ {p}\, \mu(1)\mathop{\sideset{}{^\star} \sum_{\beta\;\mathrm{mod} \; p}\, \, \sideset{}{^\star} \sum_{\beta_1\;\mathrm{mod} \; p}}_{\beta \equiv -\beta_1 \mathrm{mod} \; p }\,e\left(\frac{\bar{\beta}m_2+\bar{\beta_1}m'_2}{p}\right)\\
 =&\,- \,\,\sideset{}{^\star} \sum_{\beta\;\mathrm{mod} \; p}\, e\left(\frac{\bar{\beta}m_2}{p}\right)\, \sideset{}{^\star} \sum_{\beta_1\;\mathrm{mod} \; p}\, e\left(\frac{\bar{\beta_1}m'_2}{p}\right) + {p}\,\,\sideset{}{^\star} \sum_{\beta_1\;\mathrm{mod} \; p}\, e\left(\frac{(m'_2-m_2)\bar{\beta_1}}{p}\right)\\
 =&\,\,-\sum_{d_1 \mid (p, m_2)} d_1\, \mu \left( \frac{p}{d_1}\right)\,\sum_{d_2 \mid (p, m'_2)} d_2\, \mu \left( \frac{p}{d_2}\right) +\,p\,\sum_{d_3 \mid (p, m'_2- m_2)} d_3\, \mu \left( \frac{p}{d_3}\right)\\
=&\,-1 + p\,\delta (m_2 \equiv 0\, \mathrm{mod} \; p ) + p\,\delta (m'_2 \equiv 0\, \mathrm{mod} \; p ) - p^2\,\delta (m_2, m'_2 \equiv 0\, \mathrm{mod} \; p )\\
&\,\,\, -p + p^2\,\delta (m'_2 \equiv m_2\, \mathrm{mod} \; p ).
 \end{align*}
Notice that the last term in the above line will be dominating. Hence for the case when $q_2 = p$, we can write

\begin{align*}
\mathfrak{C}(q_2) \ll \, q^2_2 \, \delta (m_2 \equiv m_2^\prime\, \mathrm{mod} \; q_2 ).  
\end{align*}
	  Now, when we follow the same steps as above for $q_2 = p^{\gamma}$ with $\gamma > 1$, we have either $ d = p^\gamma$ or $ d=p^{\gamma-1}$.  For the case when $ d = p^{\gamma-1}$, from  $  \beta+ \beta_1 \equiv 0\, \mathrm{mod} \;p^{\gamma-1} $, we write $ \beta = - \beta_1 + c_1 p^{\gamma-1} = - \beta_1 (1 - c_1\overline{\beta_1} p^{\gamma-1})$ , where $ c_1$ is varying modulo $p$. Then we obtain that $ \overline{ \beta} = -\overline{\beta_1} (1+c_1 \overline{\beta_1} p^{\gamma- 1} )$.  In this case, we calculate that the character sum $\mathfrak{C}(q_2)$ will become
\begin{align*}
 \mathfrak{C}(q_2 = p^{\gamma}) \,= \, &\,p^{\gamma}\,\mu(1)\, \mathop{\sideset{}{^\star} \sum_{\beta\;\mathrm{mod} \; p^{\gamma}}\, \,   \sideset{}{^\star} \sum_{\beta_1\;\mathrm{mod} \; p^{\gamma}}\,}_{\beta \equiv -\beta^\prime \mathrm{mod} \; p^{\gamma}} e\left(\frac{\bar{\beta}m_2+\bar{\beta_1}m'_2}{p^{\gamma}}\right)\\
 &\,\,\,\,\,+\,\,\mu(p)\,\, p^{\gamma-1} \mathop{\sideset{}{^\star} \sum_{\beta\;\mathrm{mod} \; p^{\gamma}}\, \,   \sideset{}{^\star} \sum_{\beta_1\;\mathrm{mod} \; p^{\gamma}}\,}_{\beta \equiv -\beta^\prime\, \mathrm{mod} \; p^{\gamma -1} } e\left(\frac{\bar{\beta}m_2+\bar{\beta_1}m'_2}{p^{\gamma}}\right) \\ 
		=&\,p^{\gamma}\,\mathop{\sideset{}{^\star} \sum_{\beta\;\mathrm{mod} \; p^{\gamma}}\, } e\left(\frac{(m'_2-m_2)\bar{\beta}}{p^{\gamma}}\right)\\
  &\,\,\,\,- \, p^{\gamma-1} \mathop{\sideset{}{^\star} \sum_{\beta_1\;\mathrm{mod} \; p^{\gamma}}\, } e\left(\frac{(m'_2 - m_2)\bar{\beta_1}}{p^{\gamma}}\right) \sum_{c_1 (p)}  \ e\left(- \frac{\overline{\beta_1}^2  c_1 p^{\gamma -1} m_2}{ p^\gamma}\right).  
  \end{align*}
The sum over $\beta$ and  $\beta_1$ are the Ramanujan sums. By executing the $c_1$-sum, the character sum $\mathfrak{C}(q_2 =p^{\gamma})$ will become
  \begin{align*}
 \,&  p^{\gamma}\,\mathop{\sum_{d_1|(p^{\gamma}, m'_2-m_2)}}\,d_1 \mu\left(\frac{p^{\gamma}}{d_1}\right) -   p^{\gamma}  \ \mathop{\sum_{d_2|(p^{\gamma}, m'_2-m_2)}}\,d_2 \mu\left(\frac{p^{\gamma}}{d_2}\right)\delta (\overline{\beta_1}^2\,m_2 \equiv 0\, \mathrm{mod} \; p ).\end{align*}
 Again, executing the remaining sums, with $d_i = p^{\gamma} \,\text{or}\,\, p^{\gamma-1},\, i = 1,2$, we get that the first term in the above line will dominate and with $q_2 = p^{\gamma}$, we finally have
 \begin{align}\label{q3}
     \mathfrak{C}(q_2) \ll \,\, q^2_2\,  \delta (m_2 \equiv m_2^\prime \, \mathrm{mod} \; q_2 ).
 \end{align}
By plugging the estimates from equations \eqref{q2} and \eqref{q3} into equation \eqref{q1}, we get
 \begin{align*}
  \mathfrak{C}(q; 0, 0) \, \ll \,\frac{q^5}{n^2_1\,m_1}\,\, \delta (m_2  \equiv m_2^\prime \, \mathrm{mod} \; q_2 ),
 \end{align*}
which is our desired result.
\end{proof}
\vspace{1cm}

\noindent In the next lemma, we will prove the cancellations in a general character sum, which further helps us to get the required bound for our character sum in the non-zero frequency cases $n_2 \neq 0, h \neq 0$ and $n_2 \neq 0, h = 0$.

\begin{lemma}\label{FK}
Consider the following character sum
\begin{align*}
 \mathfrak{C}(...) \,=&\,\,\sideset{}{^\star} \sum_{x\;\mathrm{mod} \; p}\,\,e\left(\frac{xc}{p}\right)\,S\left(x, (x+h)\,c_1; p\right)\, S\left( x,   mc_2; p\right)
\,\,S\left( x+h,  nc_2; p\right),      
\end{align*}
where $c,\, c_1,\, c_2,\, m,\, n,\,$$h$ are non-zero integers and $p$ is a prime number. Then we have
\begin{align*}
 \mathfrak{C}(...) \,\ll\, p^2 .   
\end{align*}
\end{lemma}
\begin{proof}
See Section \ref{PF}.
\end{proof}
\vspace{0.3cm}

\begin{lemma}\label{M3}
Let $n_2 \neq 0$ and $h = 0$ or $h \neq 0$. Also, let $q = q_1q'_2q''_2$ such that $n_1|q_1|(n_1m_1)^{\infty}$ with $(q_1, q/q_1) = 1$, and $q_2 = q/q_1 =  q'_2q''_2$, such that $q'_2$ is square free and $q''_2$ is square full with $(q'_2, 2q''_2) = 1$. Then for any $\epsilon > 0$, we have
\begin{align*}
\mathfrak{C}(q;h,n_2)\,\ll\, \frac{q^4}{n^{3/2}_1m_1}\, \,{(q_1q''_2)^{1/2}}\,\,q^{\epsilon}.    
\end{align*}
And the same bound is true for $\mathfrak{C}(q;0,n_2)\,$
 \end{lemma}

\begin{proof}
We prove the result in the case when $n_2 \neq 0, h \neq 0$, and the same steps will also work for the other case $n_2 \neq 0, h = 0$. Recall from equation \eqref{M1}, for the case $n_2 \neq 0, h \neq 0$, we have
\begin{align*}
\mathfrak{C}(q;h,n_2)
	=&\,\frac{q^2}{n_1}\,\sideset{}{^\star} \sum_{a\;\mathrm{mod} \; q}\,\, \,\sideset{}{^\star} \sum_{\alpha\;\mathrm{mod} \; q/n_1}\,e\left(\frac{\alpha a- (\overline{\bar{\alpha}+n_2})(a+h)}{q/n_1}\right)  \notag\\
	&\times\,S\left( a, m_{2}; q/m_{1}\right)\, \,S\left( a+h,   m'_{2}; q/m_{1}\right).    
\end{align*}
We make the change of variables
\begin{align*}
	x = n_2 \alpha +1  \Longrightarrow \alpha =  (x- 1) \overline{n_2}. 
\end{align*}
Here we are using $(n_2, q/n_1) = 1$. In the other case, when $(n_2, q/n_1) \neq 1$, we can get extra savings in the $n_2 \ll N_2$ length and prove even a better bound for the character sum $\mathfrak{C}(q;h,n_2)$. So, we will work out the details only for the case when $(n_2, q/n_1) = 1$. Now we have
\begin{align*}
	\overline{\alpha} + n_2 = \overline{\alpha} (1 + \alpha n_2) =   x n_2 \overline{(x- 1)}\\
 \Longrightarrow \,\,- ( \overline{\bar{\alpha}+n_2} ) = -  (x-1)  \overline{ x n_2}  =  \overline{ x n_2} -  \overline{  n_2}. 
\end{align*}   
By plugging the values of $\alpha$ and $( \overline{\bar{\alpha}+n_2} )$ in the above expressions of the character sum, we can write
\begin{align*}
	\mathfrak{C}(q; h, n_2) 
 =&\, \frac{q^2}{n_1}\,\sideset{}{^\star} \sum_{a\;\mathrm{mod} \; q}\,\,e\left(-\frac{(2a+h)\overline{n_2}}{q/n_1}\right)\,\,S\left( a,   m_{2}; q/m_{1}\right)\\
& \times\,\,S\left( a\overline{n_2},  (a+h)\overline{n_2}; q/n_{1}\right)\,\,S\left( a+h, m'_{2}; q/m_{1}\right). \\
\end{align*}
Write the reduced residue class $a$ modulo $q$  as $ a = a_1q_2 \overline{q_2} + a_2 q_1\,\overline{q_1}$, where $a_1$ and $a_2$ are the reduced residue classes modulo $q_1$ and $q_2$, respectively. Also, using the properties of Kloosterman sums from Lemma \ref{B30}, we can rewrite our character sum as
\begin{align}\label{m27}
 \mathfrak{C}(q; h, n_2) \,=\, \frac{q^2}{n_1}\,e\left(-\frac{h\overline{n_2}}{q/n_1}\right)\,\mathfrak{C}_1(...)\,\mathfrak{C}_2(...),
\end{align}
where
\begin{align*}
 \mathfrak{C}_1(...) =&\, \,\,\sideset{}{^\star} \sum_{a_1\;\mathrm{mod} \; q_1}\,\,e\left(-\frac{2a_1\overline{n_2}}{q_1/n_1}\right)\,\,S\left( a_1,  (a_1+h)(\overline{n_2q_2})^2; q_1/n_{1}\right)\\
 & \,\,\,\,\times\, S\left( a_1,   m_{2}\overline{q_2}^2; q_1/m_{1}\right)
\,\,S\left( a_1+h,   m'_{2}\overline{q_2}^2; q_1/m_{1}\right),
 \end{align*}
 and
 \begin{align*}
  \mathfrak{C}_2(...) =&\,\,\sideset{}{^\star} \sum_{a_2\;\mathrm{mod} \; q_2}\,\,e\left(-\frac{2a_2\overline{n_2}}{q_2}\right)\,\,S\left( a_2,  (a_2+h)(\overline{q_1n_2/n_1})^2; q_2\right)\\
 & \,\,\,\times\, S\left( a_2,   m_{2}(\overline{q_1/m_1})^2; q_2\right)
\,\,S\left( a_2+h,   m'_{2}(\overline{q_1/m_1})^2; q_2\right).  
\end{align*}
Using Weil’s bound for Kloosterman sums and summing over $a_1$ trivially in $\mathfrak{C}_1(...)$, we get
\begin{align*}
\mathfrak{C}_1(...) \,& \ll \, \frac{(q_1)^{5/2}}{m_1\sqrt{n_1}}\,d\left(\frac{q_1}{n_1}\right)\,d\left(\frac{q_1}{m_1}\right)\,d\left(\frac{q_1}{m_1}\right)\\
\,&\,\,\,\,\,\times\left(a_1,  (a_1+h)(\overline{n_2q_2})^2, \frac{q_1}{n_1}\right)^{1/2}
 \left(a_1,  m_{2}\overline{q_2}^2, \frac{q_1}{m_1}\right)^{1/2}\left(a_1+h,  m'_{2}\overline{q_2}^2, \frac{q_1}{m_1}\right)^{1/2}.
\end{align*}
Since we know that when we take an average over $h, m_2, m'_2$ and $n_2$, all these greatest common divisors will behave like $1$ (see Lemma \ref{ll1}). Hence, we can write
\begin{align}\label{m26}
\mathfrak{C}_1(...) \,\ll \,\frac{q_1^{5/2}}{m_1\sqrt{n_1}}\,\,(q_1)^{\epsilon},
\end{align}
where $\epsilon > 0$ is arbitrarily small, which comes into the role due to the estimate of divisor function, i.e. $d(n) \ll\, n^{\epsilon}$.
Now, to analyze the character sum $\mathfrak{C}_2(...)$ and find its bound, we further break $q_2$ to extract its square full part. We write $q_2 = q'_2q''_2$ such that $(q'_2, 2q''_2) =1$, where $q'_2$ is square free and $4q''_2$ is square full. In this situation, we can rewrite the character sum $\mathfrak{C}_2(...)$ as
\begin{align}\label{m24}
 \mathfrak{C}_2(...) = \,\,\mathfrak{C}'_2(...)\,\mathfrak{C}''_2(...),  \hspace{1cm} \text{where}, 
\end{align}
\begin{align*}
\mathfrak{C}'_2(...) = &\,\,\,\sideset{}{^\star} \sum_{a'_2\;\mathrm{mod} \; q'_2}\,\,e\left(-\frac{2a'_2\overline{n_2}}{q'_2}\right)\,\,S\left( a'_2,  (a'_2+h)(\overline{q_1q''_2n_2/n_1})^2; q'_2\right)\\
 & \,\,\,\times\, S\left( a'_2,   m_{2}(\overline{q_1q''_2/m_1})^2; q'_2\right)
\,\,S\left( a'_2+h,   m'_{2}(\overline{q_1q''_2/m_1})^2; q'_2\right), \end{align*}
and
\begin{align*}
\mathfrak{C}''_2(...) = &\,\,\sideset{}{^\star} \sum_{a''_2\;\mathrm{mod} \; q''_2}\,\,e\left(-\frac{2a''_2\overline{n_2}}{q''_2}\right)\,\,S\left( a''_2,  (a''_2+h)(\overline{q_1q'_2n_2/n_1})^2; q''_2\right)\\
 &\,\,\, \times\, S\left( a''_2,   m_{2}(\overline{q_1q'_2/m_1})^2; q''_2\right)
\,\,S\left( a''_2+h,   m'_{2}(\overline{q_1q'_2/m_1})^2; q''_2\right).    
\end{align*}
Estimating $\mathfrak{C}''_2(...)$ by using Weil’s bound for
Kloosterman sum (with Lemma \ref{ll1}) and summing over $a''_2$ trivially, we get 
\begin{align*}
\mathfrak{C}''_2(...) \, \ll \,\, \left(q''_2\right)^{5/2+\epsilon}.    
\end{align*}
We now proceed to estimate $\mathfrak{C}'_2(...) $. Since we have chosen $q'_2$ as square free, its prime factorization looks like $q'_2 = p_1p_2....p_k,\, k \geq 1$, we can rewrite $\mathcal{C}'_2(...)$ as
\begin{align}\label{m23}
 \mathfrak{C}'_2(...) = \prod_{1 \leq i \leq k}\, \mathfrak{C}'_{2,i}(...), \hspace{0.2cm} \text{where}   
\end{align}
\begin{align*}
\mathfrak{C}'_{2,i} \,=&\,\,\sideset{}{^\star} \sum_{x\;\mathrm{mod} \; p_i}\,\,e\left(-\frac{2x\overline{n_2}}{p_i}\right)\,\,S\left( x,  (x+h)\,(\overline{q_1q''_2p'_in_2/n_1})^2; p_i\right)\\
 & \,\,\,\times\, S\left( x,   m_{2}(\overline{q_1q''_2p'_i/m_1})^2; p_i\right)
\,\,S\left( x+h,   m'_{2}(\overline{q_1q''_2p'_i/m_1})^2; p_i\right),   
\end{align*}
where $p'_i = q'_2/p_i$. Define $s_1 = (\overline{q_1q''_2p'_in_2/n_1})^2,$ $s_2 =  (\overline{q_1q''_2p'_i/m_1})^2$, where $(s_j, p_i) = 1$ for all $1 \leq j \leq 2,\, \, 1 \leq i \leq k$. We can rewrite the above character sum as
\begin{align}\label{m22}
 \mathfrak{C}'_{2,i} \,=&\,\,\sideset{}{^\star} \sum_{x\;\mathrm{mod} \; p_i}\,\,e\left(-\frac{2x\overline{n_2}}{p_i}\right)\,\,S\left( x,  (x+h)\,s_1; p_i\right)\notag\\
 & \,\,\,\times\, S\left( x,   m_{2}s_2; p_i\right)
\,\,S\left( x+h,   m'_{2}s_2; p_i\right).      
\end{align}
 The square root cancellations in the character sum of type \eqref{m22} have already been established in the previous Lemma \ref{FK}. So, we get
 \begin{align*}
  \mathfrak{C}'_{2,i} \,\ll \,\,(p_i)^2. 
 \end{align*}
Using the above bound in the equation \eqref{m23}, we get
\begin{align*}
 \mathfrak{C}'_2(...) \, \ll \, \,(q'_2)^2 .      
\end{align*}
From equation \eqref{m24}, we get
\begin{align}\label{m25}
\mathfrak{C}_2(...)\,\ll\, (q'_2)^2 \,(q''_2)^{5/2} .    
\end{align}
Finally, by plugging the estimates from \eqref{m26} and \eqref{m25} into equation \eqref{m27}, we get
\begin{align*}
 \mathfrak{C}(q;h,n_2)\,\ll\, \frac{q^4}{ n^{3/2}_1m_1}\, \,{(q_1q''_2)^{1/2}}\,\,q^{\epsilon}.    
\end{align*}
 This was our desired result.
\end{proof}
\vspace{0.3cm}

\noindent Now we consider the last remaining case $n_2 = 0$ and $h \neq 0$. We have the following lemma.
\vspace{0.3cm}

\begin{lemma}\label{M4}
Letting $n_2 = 0, h \neq 0$ and using the same factorisation of $q$ as in Lemma \ref{M3} above, we have
 \begin{align*}
\mathfrak{C}(q;h,0)\, \ll \, \frac{q^4}{n^{3/2}_1m_1}\, \,{(q_1q''_2)^{1/2}}\,\,q^{\epsilon} .    
 \end{align*}
\end{lemma}
\begin{proof}
In the present case, our character sum given in equation \eqref{M1} will become
\begin{align*}
\mathfrak{C}(q;h,0)
	=&\,\,\frac{q^2}{n_1}\,\,\sideset{}{^\star} \sum_{\alpha\;\mathrm{mod} \; q/n_1}\,e\left(\frac{-\alpha h}{q/n_1}\right)\,\sideset{}{^\star} \sum_{a\;\mathrm{mod} \; q}\,\,   S\left( a,   m_{2}; q/m_{1}\right) \,S\left( a+h,   m'_{2}; q/m_{1}\right)\\
 =&\,\,\frac{q^2}{n_1}\, \mathfrak{C}_1(q;h,0)\,S(-h, 0,;q/n_1),
\end{align*}
where $S(-h, 0,;q/n_1)$ is the Ramanujan sum and
\begin{align*}
\mathfrak{C}_1(q;h,0) \, = \, \sideset{}{^\star} \sum_{a\;\mathrm{mod} \; q}\,\,   S\left( a, m_{2}; q/m_{1}\right) \,S\left( a+h,   m'_{2}; q/m_{1}\right).  
\end{align*}
For the character sum  $\mathfrak{C}_1(q;h,0)$, by proceeding with the same steps as in the above Lemma \ref{M3}, we can get our desired bound. By taking an average over $q \leq Q$ with the Ramanujan sum $S(-h, 0, q/n_1)$, we save the whole length $q$. So, In this case, we get a better bound than the last two cases, but let's stick to the same bound as in Lemma \ref{M3}. 
\end{proof}
\vspace{0.3cm}

\subsection{Estimation of zero frequency} In this subsection, we will estimate $\Omega$ and our main sum $\mathcal{D}_{1}(H,X)$ for zero frequency case $ n_2= 0 = h$. Let  $\Omega_0$ and  $\mathcal{D}^0_{1}(H,X)$ denote the contribution of zero frequency to $\Omega$ and  $\mathcal{D}_{1}(H,X)$. We have the following lemma.
\vspace{0.2cm}

\begin{lemma}\label{B18}
	For the case $n_2 = 0 = h$, we have
	 
	\begin{align*}
		\mathcal{D}^0_{1}(H,X) \ll \frac{Q^3}{\,H}.
	\end{align*} 
\end{lemma}
\begin{proof}
	From equation \eqref{B15}, we have
	\begin{align*}
		\Omega \ll\,X^{\epsilon}\,\mathop{\mathop{\text{sup}}_{N_1 \ll N_0}}_{H_1 \ll H_0}\, &\frac{H_1\,N_1\,n_1}{q^2}\,\sum_{m_{1}|q} m_1^{2/3}\,\sum_{m_{2}\ll M_0} \frac{|A_{\pi_2}(m_{1},m_{2})|}{m^{1/3}_2}\,\,\sum_{m'_{2}\ll M_0}\frac{|A_{\pi_2}(m_{1},m'_{2})|}{(m'_2)^{1/3}}\notag\\
 &\times \,\sum_{h\in \mathbb{Z}}\,\sum_{n_2 \in \mathbb{Z}} \, |\mathfrak{C}(q;h,n_2)|\,|\mathcal{J}(h,n_2)|.
	\end{align*} 
For the zero frequency case $n_2 =0$ and $h=0$, we get
	\begin{align*}
		\Omega_0  \ll& \,X^{\epsilon}\,\mathop{\mathop{\text{sup}}_{N_1 \ll N_0}}_{H_1 \ll H_0}\,\frac{H_1\,N_1\,n_1}{q^2} \sum_{m_{1}|q} m_1^{2/3}\\
  \times &\mathop{\sum_{m_{2}\ll M_0} \,\sum_{m'_{2}\ll M_0}}\,\frac{|A_{\pi_2}(m_{1},m_{2})|}{m^{1/3}_2}\,\frac{|A_{\pi_2}(m_{1},m'_{2})|}{ (m'_2)^{1/3}}|\mathfrak{C}(q;0,0)|\,|\mathcal{J}(0,0)|.
	\end{align*}
Plugging in the bound for the character sum $\mathfrak{C}(q;0,0)$ and integral transform $\mathcal{J}(0,0)$ from Lemma \ref{B17} and Lemma \ref{f1}, respectively, we get  
 \begin{align*}
		\Omega_0  \ll& \,X^{\epsilon}\,\mathop{\mathop{\text{sup}}_{N_1 \ll N_0}}_{H_1 \ll H_0}\,\frac{H_1\,N_1\,n_1}{q^2}\,\frac{q^5}{n^2_1}\, \frac{q^3}{Q^2 (XN_1n^2_1)^{1/3}}\,B(m_2,m'_2), \hspace{1cm} \text{where}
  \end{align*} 
  \begin{align*}
B(m_2,m'_2) =\, \sum_{m_{1}|q|(m_1)^{\infty}} \frac{1}{m_1^{1/3}}\,\mathop{\sum_{m_{2}\ll M_0} \sum_{m'_{2}\ll M_0}}_{m_2 \equiv m_2^\prime \mathrm{mod} \; q_2}\frac{|A_{\pi_2}(m_{1},m_{2})|}{m^{1/3}_2}\frac{|A_{\pi_2}(m_{1},m'_{2})|}{ (m'_2)^{1/3}}.    
  \end{align*}
  Now we will find a bound for $B(m_2,m'_2)$. For convenience, we divide the $m_2$ and $m'_2$ sums into dyadic blocks $m_2,\,m'_2  \sim M_1$, where $M_1 \ll M_0$. Now we can write
  \begin{align*}
  B(m_2,m'_2) \,\ll\, X^{\epsilon}\mathop{\text{sup}}_{M_1 \ll M_0}\,\,(S_1 + S_2)\,\frac{1}{M^{2/3}_1},  \hspace{1cm} \text{where}  
  \end{align*}
  \begin{align*}
 S_1 \,= \,\,\sum_{m_{1}|q_1|(m_1)^{\infty}} \,\frac{1}{m_1^{1/3}}\, \mathop{\sum_{m_{2}\sim M_1} \sum_{m'_{2}\sim M_1}}_{m_2 \equiv m_2^\prime \mathrm{mod} \; q_2} \,|A_{\pi_2}(m_{1},m_{2})|^2.  
\end{align*}
Similarly, we can write the expression for $S_2$. Now, we will estimate $S_1$, and the same steps will also work to estimate $S_2$. We can rewrite $S_1$ as
\begin{align*}
 S_1 \,= \,\sum_{m_{1}|q_1|(m_1)^{\infty}}\,\frac{1}{m_1^{1/3}}\, \mathop{\sum_{m_{2}\sim M_1} } \,|A_{\pi_2}(m_{1},m_{2})|^2\,\mathop{\sum_{m'_{2}\sim M_1}}_{{m'_2 \equiv m_2 \mathrm{mod} \; q_2}}\,1.  
\end{align*}
Executing the $m'_2$-sum by counting $m'_2$ with the help of the congruence condition and $m_2$-sum by using Lemma \ref{RMB}, we get
\begin{align*}
    S_1 \,\ll&\, \,\sum_{m_{1}|q_1|(m_1)^{\infty}} \,\frac{q_1}{m_1^{1/3}}\,m^{2\theta +\epsilon}_1\,M_1\,\frac{M_1}{q},
\end{align*}
where $\theta \leq 5/14$. Since the sum $S_2$ is similar to $S_1$, we can write
\begin{align*}
  B(m_2,m'_2) \ll X^{\epsilon}\mathop{\text{sup}}_{M_1 \ll M_0} \sum_{m_{1}|q_1|(m_1)^{\infty}} \,\frac{q_1}{m_1^{1/3}}\,m^{2\theta +\epsilon}_1 \frac{M^{4/3}_1}{q} \ll X^{\epsilon}\,\sum_{m_{1}|q_1|(m_1)^{\infty}} \,\frac{q_1}{m_1^{1/3}}\,m^{2\theta +\epsilon}_1 \frac{M^{4/3}_0}{q}.  
\end{align*}
Now, using $M_0 = M/m^2_1$ and $\theta \leq 5/14$, we get
\begin{align*}
  B(m_2,m'_2) \, \ll  \,\frac{M^{4/3}}{q}. 
\end{align*}
Finally, by plugging this bound into the expression for $\Omega_0$, we get
 \begin{align*}
\Omega_0   \,\ll\, \frac{H_0\,N^{2/3}_0\,M^{4/3}\, q^5}{Q^2\,X^{1/3}\,n^{5/3}_1} .
	\end{align*}  Now, from equations \eqref{A32} and \eqref{c32}, we have
	\begin{align}\label{b10}
		\mathcal{D}^0_{1}(H,X) &\ll \frac{X^{4/3}\,H_0^{1/4}}{QH^{1/4}}\,\,\sum_{q \leq Q}\,\frac{1}{q^{7/2}}\,\sum_{n_1| q}\,\,n^{1/3}_1 \,\Theta^{1/2}(\Omega_0)^{1/2}.
	\end{align} 
Plugging in our bound of $\Omega_0$ and using $ N_0 = N/n^2_1$, we get that our main sum $\mathcal{D}^0_{1}(H,X) $ is dominated by
	\begin{align*}
	&\, \frac{X^{4/3}\,H_0^{1/4}}{QH^{1/4}}\,\sum_{q \leq Q}\,\frac{1}{q^{7/2}}\,\sum_{n_1 |q} n^{1/3}_1\, \left(\frac{H_1\,N^{2/3}_1\,M^{4/3}\, q^5}{Q^2\,X^{1/3}\,n^{5/3}_1}  \right)^{1/2}\,\Theta^{1/2}\\
		&\ll \frac{X^{7/6} \,N^{1/3}\,H_0^{3/4}\,M^{2/3}}{Q^2\,H^{1/4}\,}\,\,\sum_{n_1 \leq q}\,\frac{1}{n^{7/6}_1}\,\Theta^{1/2}.
  \end{align*}
For $k \geq 7/6$, using the value of $\Theta$ from equation \eqref{c31}, it is a simple observation that by applying Cauchy-Schwarz inequality to the $n_1$-sum, we can get
\begin{align}\label{b12}
\sum_{n_1 \ll q}\frac{\Theta^{1/2}}{n^{k}_1} \,\ll\, N^{1/6}.    
\end{align}
Using this result, we get
  \begin{align*}
\mathcal{D}^0_{1}(H,X) &\,\ll\, \frac{X^{7/6} \,N^{1/2}\,H_0^{3/4}\,M^{2/3}}{Q^2\,H^{1/4}}.
	\end{align*} Now using
	\begin{align}\label{f2}
		N \, \ll \,\frac{Q^3}{X}, \hspace{0.5cm} M \, \ll\, \frac{Q^3}{X}, \hspace{0.5cm} H_0 \,\ll\, \frac{Q^2}{H},
	\end{align} we finally get
	\begin{align*}
		\mathcal{D}^0_{1}(H,X) \ll \frac{Q^3}{H}.
	\end{align*} This is our desired result. 
\end{proof}

\vspace{0.1cm} 
\subsection{Estimation of non-zero frequency} There are a total of three possible cases of non-zero frequencies i.e. when at least one of $n_2$ and $h$ is non-zero. In this subsection, we will estimate $\Omega$ given in \eqref{B15} for all these three cases and denote their sum by $\Omega_{\neq 0}$. Then we will estimate our principal sum  $\mathcal{D}_1(H, X)$ given in equation \eqref{A32} in all these cases and we denote it by $\mathcal{D}^{\neq 0}_{1}(H,X)$.

\vspace{0.2cm}

\begin{lemma}\label{B22}
	For the non-zero frequency cases, we have
	\begin{align*}
		\mathcal{D}^{\neq 0}_{1}(H,X) \ll \frac{X^{1/2}\,Q^{3/2}}{\,H^{1/2}} +  \frac{X^{1/2}\,Q^{2}}{\,H} +  \frac{Q^{5/2}}{\,H^{1/2}}.
	\end{align*}
\end{lemma}
\begin{proof}
From equation \eqref{B15}, we have 
  \begin{align*}
  \Omega \,\ll&X^{\epsilon}\,\mathop{\mathop{\text{sup}}_{N_1 \ll N_0}}_{H_1 \ll H_0}\,\frac{H_1\,N_1\,n_1}{q^2}\sum_{m_{1}|q}\,m^{2/3}_1  \,\sum_{m_{2}\ll M_0} \frac{|A_{\pi_2}(m_{1},m_{2})|}{ m^{1/3}_{2}}\,\sum_{m'_{2}\ll M_0}\frac{|A_{\pi_2}(m_{1},m'_{2})|}{ (m'_{2})^{1/3}}\notag\\
 &\times \,\sum_{0\leq |h| \leq H_2}\,\sum_{0\leq |n_2| \leq N_2} \, |\mathfrak{C}(q;h,n_2)|\,|\mathcal{J}(h,n_2)|.    
  \end{align*}
As $\Omega_{\neq 0}$ stands for the sum of contribution of non-zero frequency in all three cases, it is dominated by 
\begin{align*}
 &X^{\epsilon}\,\mathop{\mathop{\text{sup}}_{N_1 \ll N_0}}_{H_1 \ll H_0}\frac{H_1N_1 n_1}{q^2} \,\left(\sum_{0 < |h| \leq H_2}\,\sum_{0< |n_2| \leq N_2} \, |\mathfrak{C}(q;h,n_2)|\,|\mathcal{J}(h,n_2)| \right. \\
 &\,\,\,\,\,\, \left. + \sum_{0 <|n_2| \leq N_2} \, |\mathfrak{C}(q;0,n_2)|\,|\mathcal{J}(0,n_2)| + \sum_{0< |h| \leq H_2}\, \, |\mathfrak{C}(q;h,0)|\,|\mathcal{J}(h,0)|\right)\,C(m_2,m'_2),
\end{align*}
where $C(m_2, m'_2)$ is given by
\begin{align*}
 \,\sum_{m_{1}|q_1|(n_1m_1)^{\infty}}\,\,  \sum_{m_{2}\ll M/m^2_1} \frac{|A_{\pi_2}(m_{1},m_{2})|}{ m^{1/6}_{1}m^{1/3}_{2}}\,\sum_{m_{1}|q_1|(n_1m_1)^{\infty}}\,\,  \sum_{m'_{2}\ll M/(m'_1)^2} \frac{|A_{\pi_2}(m_{1},m'_{2})|}{ m_1^{1/6}(m'_2)^{1/3}}.      
 \end{align*}
Now, we use Lemma \ref{M3}, Lemma \ref{M4} for the bound of character sum, and Lemma \ref{f1} for the integral transform in all three cases. Also, estimating the $n_2$- and $h$-sum trivially using their dual length given in Lemma \ref{f1}, we notice that $\Omega_{\neq 0}$ is dominated by 
  \begin{align*}
 \mathop{\mathop{\text{sup}}_{N_1 \ll N_0}}_{H_1 \ll H_0}\frac{H_1N_1q^2 (q_1q''_2)^{1/2}}{n^{1/2}_1}\left(\frac{q}{H_1}\frac{(XN)^{1/3}}{n_1N_1}\frac{q^2}{Q^2} + \frac{(XN)^{1/3}}{n_1N_1}\frac{q^3}{Q^2\sqrt{HH_1}} + \frac{q}{H_1}\frac{q^2}{Q^2}\right)C(m_2,m'_2),
	\end{align*} 
 where we are writing $q = q_1q'_2q''_2$ as in Lemma \ref{M3}. Note that
\begin{align*}
 &\sum_{m_{1}|q_1|(n_1m_1)^{\infty}}\,\,\sum_{m_{2}\ll M/m^2_1} \frac{|A_{\pi_2}(m_{1},m_{2})|}{ m^{1/6}_{1}m^{1/3}_{2}} \,\ll\,  \sum_{m_{1}\ll q}\,\sum_{m_{2}\ll M/m^2_1} \frac{|A_{\pi_2}(m_{1},m_{2})|}{ m^{1/6}_{1}m^{1/3}_{2}}\\
 &\ll\left(\sum_{m_{1}\ll q}\,\,\sum_{m_{2}\ll M/m^2_1} |A_{\pi_1}(m_1,m_2)|^2\right)^{1/2}\left(\sum_{m_{1}\ll q}\frac{1}{m^{1/3}_1}\sum_{m_{2}\ll M/m^2_1}\frac{1}{m^{2/3}_2}\right)^{1/2} \ll M^{2/3}, 
\end{align*} 
where we have used Lemma \ref{ramanubound} for the estimate of the first factor. By using the above estimates for the $m_2$-sum and similarly, for $m'_2$-sum, we get
 \begin{align*}
 C(m_2,m'_2) \,\ll\,M^{4/3}.       
 \end{align*}
 Hence, we get
 \begin{align*}
 \Omega_{\neq 0} \,\ll\, \frac{M^{4/3}\,q^5 (q_1q''_2)^{1/2}}{Q^2}\left(\frac{(XN)^{1/3}}{n^{3/2}_1} + \frac{(XN)^{1/3}\,H^{1/2}_0}{n^{3/2}_1\,\sqrt{H}} + \frac{N_0}{n^{1/2}_1}\,\right).    
 \end{align*}
Now, from equations \eqref{A32} and \eqref{c32}, we have
	\begin{align*}
		\mathcal{D}^{\neq 0}_{1}(H,X) &\ll \frac{X^{4/3}\,H_0^{1/4}}{QH^{1/4}}\,\,\sum_{q \leq Q}\,\frac{1}{q^{7/2}}\,\sum_{n_1| q}\,\,n^{1/3}_1 \,\Theta^{1/2}(\Omega_0)^{1/2}.
	\end{align*} 
 Using the above bound of $\Omega_{\neq 0}$ and Lemma \ref{ll2} for the estimation of the $q''_2$-sum, we obtain that $\mathcal{D}^{\neq 0}_{1}(H,X)$ is dominated by
	\begin{align*}
	&\frac{X^{4/3}\,M^{2/3}\,H_0^{1/4}}{Q^2\,H^{1/4}}\,\,\sum_{n_1 \ll q}\,\frac{\Theta^{1/2}}{n^{-1/3}_1}\,\left(\frac{(XN)^{1/3}}{n^{3/2}_1} + \frac{(XN)^{1/3}\,H^{1/2}_0}{n^{3/2}_1\,\sqrt{H}} + \frac{N}{n^{5/2}_1}\,\right)^{1/2}\\
 &\,\,\,\,\,\,\,\times\,\mathop{\sum_{n_1|q_{1}| (n_1m_1)^{\infty} }}\,\mathop{\sum_{q'_2 \leq Q/q_1}}_{q'_2 \text{is square free}}\,\mathop{\sum_{q''_2 \leq Q/q_1q'_2}}_{4q''_2 \text{is square full}}\,\frac{(q_1q''_2)^{1/4}}{(q_1q'_2q''_2)}\\
 &\ll\,\frac{X^{4/3}M^{2/3}H_0^{1/4}}{Q^2\,H^{1/4}}\left((XN)^{1/6}\sum_{n_1 \ll q}\frac{\Theta^{1/2}}{n^{5/12}_1} + \frac{(XN)^{1/6}H^{1/4}_0}{H^{1/4}}\sum_{n_1 \ll q}\frac{\Theta^{1/2}}{n^{5/12}_1} + N^{1/2} \sum_{n_1 \ll q}\frac{\Theta^{1/2}}{n^{11/12}_1}\right)\\
 &\,\,\,\,\,\,\,\times\,\mathop{\sum_{n_1|q_{1}| (n_1m_1)^{\infty} }}\,\frac{1}{q_1^{3/4}}\,\mathop{\sum_{q'_2 \leq Q/q_1}}_{q'_2 \text{is square free}}\,\frac{1}{q'_2}\,\\
& \ll\,\frac{X^{4/3}\,M^{2/3}\,H_0^{1/4}}{Q^2\,H^{1/4}}\,\,\left( (XN)^{1/6}\sum_{n_1 \ll q}\frac{\Theta^{1/2}}{n^{7/6}_1} + \frac{(XN)^{1/6}H^{1/4}_0}{H^{1/4}}\sum_{n_1 \ll q}\frac{\Theta^{1/2}}{n^{7/6}_1} + N^{1/2} \sum_{n_1 \ll q}\frac{\Theta^{1/2}}{n^{5/3}_1}\right).
\end{align*} 
By using the result of equation \eqref{b12} for the $n_1$-sums and plugging in the values of $N,M $ and $H_0$ given in equations \eqref{f2}, we finally get
	\begin{align*}
		\mathcal{D}^{\neq 0}_{1}(H,X) \ll \frac{X^{1/2}\,Q^{3/2}}{\,H^{1/2}} +  \frac{X^{1/2}\,Q^{2}}{\,H} +  \frac{Q^{5/2}}{\,H^{1/2}}.
	\end{align*} This is our desired result.
\end{proof}
\vspace{0.5cm}

\subsection{Final Estimates} We have,
\begin{align*}
	\mathcal{D}(H,X) \,\ll\, &\mathcal{D}_1(H,X) = \mathcal{D}^{0}_1(H,X) + \mathcal{D}^{\neq 0}_{1}(H,X). 
\end{align*} By substituting the bounds of $\mathcal{D}^{0}_1(H,X)$, $D^{\neq 0}_{1}(H,X)$ from Lemma \ref{B18}, Lemma \ref{B22}, respectively, we get
\begin{align*}
	\mathcal{D}(H,X) \ll& \,\left\{\frac{Q^3}{\,H} +\frac{X^{1/2}\,Q^{3/2}}{\,H^{1/2}}  + \frac{X^{1/2}\,Q^{2}}{\,H} + \frac{Q^{5/2}}{\,H^{1/2}} \right\}.
\end{align*}
By using  $Q = \sqrt{X}$, for $\delta >0$, we have

\begin{align}
	\mathcal{D}(H,X) \ll X^{1-\delta + \epsilon}  \hspace{0.5cm}\text{whenever} \hspace{0.5cm} X^{1/2+\delta} \leq H \leq X, 
\end{align} which is our desired result. 
\vspace{0.7cm}

\section{Proof of Lemma \ref{FK}}\label{PF}
In this section, we will present the proof of Lemma \ref{FK}. Referring to Katz's book \cite{KAT} and \cite{FKM} to become acquainted with the notations and terminology utilized in the proof.

\begin{proof}
For a prime $p$ and the natural numbers $m, n$, we have, $S(m, n;p) = S(1, mn;p)$ whenever $p \nmid m$. We also write $S(1, u ;p) = \text{Kl}_2(u, p)$, where  $\text{Kl}_2(u, p)$ denote the classical Kloosterman sum in two variables which is defined as
\begin{align*}
\text{Kl}_2(u, p) = \mathop{\sum_{x y \equiv u \,(\bmod p)}}_{x, y\, \in\, \mathbb{Z}/p\mathbb{Z}}\, e\left(\frac{x+y}{p}\right).    
\end{align*}
Let $\text{Kl}_2$ be the associated Kloosterman sheaf, then $\forall x \in \mathbb{F}^{\times}_p$, $\text{tr}(Frob_{x}|\text{Kl}_2)  = \text{Kl}_2(x; p)$.
In this terminology, we want to bound the sum
\begin{align*}
\,\sideset{}{^\star} \sum_{x\;\mathrm{mod} \; p}\,\,e\left(\frac{cx}{p}\right)\,\text{Kl}_2\left( c_1x(x+h); p\right)\, \text{Kl}_2\left(c_2 x; p\right)
\,\,\text{Kl}_2\left( c_3(x+h),; p\right),    
\end{align*}
where $c, c_1, c_2, c_3 \not\equiv 0 (\bmod p)$. We know that (see \cite{KAT}) $\text{Kl}_2$ is pure of weight $0$, rank $2$, lisse outside $\{0, \infty\}$, tame at $0$, wild at $\infty$, $\text{Swan}_{\infty} = 1$ with a single break at $1/2$. Also
\begin{align*}
 G_{\text{geom}}(\text{Kl}_2) =   G_{\text{arith}}(\text{Kl}_2) = SL(2).  
\end{align*}
Let $\mathcal{L}_c$ be the Artin-Schreier sheaf attached to the character $x \longrightarrow e\left({cx}/{p}\right)$ and for $1 \leq i \leq 3$, $\mathcal{K}_i = \mathcal{F}^*_i \text{Kl}_2$ with $\mathcal{F}_1(x) = c_1x (x+h)$, $\mathcal{F}_2(x) = c_2x$, $\mathcal{F}_3(x) = c_3(x+h)$. For $1 \leq i \leq 3$ and $x \neq 0, -h$, we have 
\begin{align*}
 \text{tr}(Frob_{x}|\mathcal{K}_i) := \mathcal{K}_i(x)   =  \text{Kl}_2 ( \mathcal{F}_i(x); p).
\end{align*}
Since $\mathcal{L}_c$ is a sheaf of rank $1$. For any $1 \leq i, j \leq 3$ with $i \neq j$, we have $\mathcal{K}_i \not\simeq_{\text{geom}}  \mathcal{L}_c \otimes \mathcal{K}_j$ and this is true for every sheaf of rank $1$. Since $\mathcal{K}_1$ has $\text{Swan}_{\infty} = 2$ with unique break at $1$ (this is the pullback of $\text{Kl}_2$ by a degree two polynomial). And $\mathcal{K}_2, \mathcal{K}_3$ have $\text{Swan}_{\infty} = 1$ with unique break at $1/2$. Since the sheaf of rank $1$ has integral weight $\geq 1$. So, $ \mathcal{K}_1$ has only break at $1$ and $\mathcal{L}_c \otimes \mathcal{K}_1$ has a break at $1/2$. Since $ G_{\text{geom}}(\text{Kl}_2)= SL(2)$, then $ G_{\text{geom}}(\mathcal{K}_i)= SL(2)$ (as $SL(2)$ has a non-proper normal subgroup of index $2$). By using Goursat-Kolchin-Ribet (see \cite[$1.8$]{MKA}), we have
\begin{align*}
G_{\text{geom}}(\mathcal{K}_1 \otimes \mathcal{K}_2 \otimes \mathcal{K}_3)   = SL(2) \times   SL(2) \times   SL(2).
\end{align*}
This means, for all $x \neq 0, -h$, we have 
\begin{align*}
    \text{tr}(Frob_{x}|\mathcal{K}_1 \otimes \mathcal{K}_2 \otimes\mathcal{K}_3)  = \mathcal{K}_1(x) \times \mathcal{K}_2(x) \times \mathcal{K}_3(x).
\end{align*}
So, the sheaf $\mathcal{K}_1 \otimes \mathcal{K}_2 \otimes\mathcal{K}_3$ is geometrically irreducible of rank $2^3$. And for any rank $1$ sheaf, for instance $\mathcal{L}_c$, the sheaf $\mathcal{K}_1 \otimes \mathcal{K}_2 \otimes\mathcal{K}_3 \otimes \mathcal{L}_c$ is geometrically irreducible and ${H}^2(\mathcal{K}_1 \otimes \mathcal{K}_2 \otimes\mathcal{K}_3 \otimes \mathcal{L}_c) = 0$. By using the Grothendieck-Lefschetz trace formula, we have
\begin{align*}
\,\mathop{\sideset{}{^\star} \sum_{x\;\mathrm{mod} \; p}}_{x \neq 0, -h}\,\text{tr}(Frob_{x}|\mathcal{K}_1 \otimes \mathcal{K}_2 \otimes\mathcal{K}_3 \otimes \mathcal{L}_c)  = \sideset{}{^\star} \sum_{x\;\mathrm{mod} \; p}\,\mathcal{K}_1(x)\, \mathcal{K}_2(x)\, \mathcal{K}_3(x)\,e\left(\frac{cx}{p}\right) \ll p^2. 
\end{align*}
This is our desired result.
\end{proof}

\vspace{0.7cm}

	\section{Sketch of the proof of Theorem 2}
		This section briefly sketches the proof of Theorem \ref{th2}. The detailed calculations and steps will be almost similar to our proof of Theorem \ref{th1}. For Theorem \ref{th2}, our object of study is the following
  
	\begin{align*}
		\mathcal{L}(H,X) = \frac{1}{H}\sum_{h \sim H}\,V_1\left(\frac{h}{H}\right)\, \sum_{n \sim X}^\infty  A_{\pi_1}(1,n)  A_{\pi_2}(1,n+h)\,V_1\left(\frac{n}{X}\right) .
	\end{align*}
	An application of the DFI delta method given in Section \ref{section3} will transform our main object into
	\begin{align}\label{d2}
		\mathcal{L}(H,X) = \frac{1}{HQ}&\sum_{q\sim Q}\frac{1}{q}\,\,\, \sideset{}{^\star} \sum_{a\bmod q}\int_{\mathbb{R}}A(u)\, \psi(q,u)\, \sum_{h \sim H}  \,e\left(\frac{ah}{q}\right)\, V_1\left(\frac{h}{H}\right)\notag\\ 
		\times& \sum_{n \sim X}  A_{\pi_1}(1,n)\, e\left(\frac{an}{q}\right)\,V_2\left(\frac{n}{X}\right)\notag\\
		\times& \sum_{m \sim Y}   A_{\pi_2} (1,m)\,e\left(\frac{-am}{q}\right)\,V_3\left(\frac{m}{Y}\right)\,du,
	\end{align}  
 where $V_1, V_2$ and $V_3$ are compactly supported smooth functions supported on $[1, 2]$. Here, we also take the same generic case as in the sketch of Theorem \ref{th1} and $1 \leq H \leq X$. After applying the delta method, we need to save $X^{1+\delta}$ in the above expression of $\mathcal{L}(H, X)$ to obtain a non-trivial bound.
 Applying the Poisson summation formula given in Lemma \ref{poisson} to the sum over $h$, we roughly arrive at
	\begin{align*}
		\sum_{h \sim H}\, e\left(\frac{ah}{q }\right)\, V_1\left(\frac{h}{H}\right)= \frac{H}{Q}  \sum_{h \sim Q/H} \sum_{\alpha \bmod q}  e \left( \frac{\alpha (a+h) }{q} \right) \, \mathcal{V}_1\left(\frac{hH}{q}\right),
	\end{align*} 
where $\mathcal{V}_1\left({hH}/{q}\right)$ is an integral transform. In this step, we save
	\begin{align*}
		\text{saving} = \sqrt{\frac{\text{Initial length of summation}}{\text{Dual length of summation}}} = \sqrt{\frac{H}{Q/H}} = \frac{H}{Q^{1/2}}.
	\end{align*}
	Next, we apply the $GL(3)$-Voronoi summation formula from Lemma \ref{gl3voronoi} to the sum over $n$ and $m$. Roughly, the $n$-sum (similarly the $m$-sum) will transform into
	\begin{align*}
		\sum_{n \sim X}A_{\pi_1}(1,n)e\left(\frac{an}{q }\right)\,V_2\left(\frac{n}{X}\right) = \frac{X^{2/3}}{q} \sum_{n_{2}\sim Q^3/X}  \frac{A_{\pi_1}(n_{2},1)}{ n^{1/3}_{2}} S\left( \bar{a}, \pm n_{2}; q \right)\,\mathcal{V}_2\left(\frac{n_2}{q^3}\right),
	\end{align*}
 where $\mathcal{V}_2\left(\frac{n_2}{q^3}\right)$ is an integral transform. Our savings in the $n$- and $m$-sums will be  $X/Q^{3/2}$ and  $X/Q^{3/2}$ , respectively. After applying the summation formulae, our main object of study $\mathcal{L}(H, X)$ will transform into the following expression
	\begin{align}\label{d1}
		\mathcal{L}(H,X) =&\frac{X^{4/3}}{Q^{5}}\,\sum_{q\sim Q} \,\,  \sum_{n_{2} \sim  Q^3/X}  \frac{A_{\pi_1}(n_{2},1)}{ n^{1/3}_{2}} \, \sum_{m_{2}\sim  Q^3/Y}  \frac{A_{\pi_2}(m_{2},1)}{ m^{1/3}_{2}}\,  \\
		\notag\times &  \,\,\,\sum_{h\sim Q/H} \, \mathcal{C}(n_2,m_2, h; q)\, \mathcal{K}(....) ,
	\end{align} where $\mathcal{K}(....)$ is some integral transform, and the character sum is given by
	\begin{align*}
		\mathcal{C}( n_2,m_2, h; q) = \sideset{}{^\star} \sum_{a\bmod q}\, \sum_{\alpha \bmod q} e\left( \frac{\alpha{(a+h)}}{q}\right) S\left( \bar{a},  \pm n_{2}; q\right)  S\left( \bar{a},  \pm m_{2}; q\right) \rightsquigarrow  \, q \, \mathcal{C}_2(...),
	\end{align*} here $\mathcal{C}_2(...)$ is another character sum modulo $q$ in which we can show the square root cancellation as we have shown in the proof of Theorem \ref{th1}. So we save $\sqrt{Q}$ in the $a$ sum. Our total saving so far over the trivial bound $X^2$ is
	\begin{align*}
		\frac{X}{Q^{3/2}} \times \frac{X}{Q^{3/2}} \times \frac{H}{Q^{1/2}} \times \sqrt{Q} =  \frac{X H}{Q}.
	\end{align*}
This saving is enough to get our desired non-trivial bound for the average version of $GL(3)\times GL(3)$ SCS $\mathcal{L}(H, X)$ whenever $X^{1/2+\delta} \leq H \leq X$.
 \vspace{0.5cm}

{\bf Acknowledgement:} 
The authors are very grateful to the anonymous referee for numerous helpful comments, which significantly improved the presentation of this paper. The authors would like to thank Ritabrata Munshi for the helpful discussions. The authors want to thank Philippe Michel and Will Sawin for their invaluable assistance in dealing with the character sum. The authors also want to thank Sumit Kumar and Prahlad Sharma for many helpful discussions during the work. The authors are also thankful to the Indian Institute of Technology Kanpur for its wonderful academic atmosphere. During this work, S. K. Singh was partially supported by D.S.T. Inspire faculty fellowship no. DST/INSPIRE/$04/2018/000945$.

{}

\end{document}